\documentclass[a4paper,reqno]{amsart}

\pdfoutput=1

\usepackage[utf8]{inputenc}
\usepackage[T1]{fontenc}

\usepackage{csquotes}
\usepackage{mathalfa}
\usepackage{amsmath}
\usepackage{amsfonts}
\usepackage[T1]{fontenc}
\usepackage{listings}
\usepackage{mathtools}
\usepackage{multicol}
\usepackage[llparenthesis,rrparenthesis]{stmaryrd}
\usepackage{xspace}
\usepackage{tikz}
\usetikzlibrary{calc}
\usetikzlibrary{arrows}
\usetikzlibrary{decorations.markings}
\usetikzlibrary{shapes} 
\usetikzlibrary{fit}

\usepackage{graphics}
\usepackage{subfig}
\usepackage{MnSymbol}
\usepackage{braket}
\usepackage{float}
\usepackage{enumitem}
\usepackage{amsthm}
\usepackage{hyperref}

\usepackage{cleveref}

\theoremstyle{plain}
\newtheorem{theorem}{Theorem}[section]
\newtheorem{satz}[theorem]{Proposition}
\newtheorem{lemma}[theorem]{Lemma}
\newtheorem{mainlemma}[theorem]{Main Lemma}
\newtheorem{corollary}[theorem]{Corollary}

\theoremstyle{definition}
\newtheorem{remark}[theorem]{Remark}
\newtheorem{definition}[theorem]{Definition}

\newtheorem{notation}[theorem]{Convention}
\newtheorem{remcor}[theorem]{Remark and Corollary}

\crefname{section}{Section}{Sections}
\crefname{subsection}{Section}{Sections}
\crefname{theorem}{Theorem}{Theorems}
\crefname{satz}{Proposition}{Propositions}
\crefname{lemma}{Lemma}{Lemmata}
\crefname{mainlemma}{Main Lemma}{Main Lemmata}
\crefname{corollary}{Corollary}{Corollaries}
\crefname{remark}{Remark}{Remarks}
\crefname{definition}{Definition}{Definitions}
\crefname{conjecture}{Conjecture}{Conjectures}
\crefname{figure}{Figure}{Figures}
\crefname{remcor}{Remark and Corollary}{Remarks and Corollaries}
\crefname{notation}{Convention}{Conventions} 

\newcommand{\convexpath}[2]{
[   
    create hullnodes/.code={
        \global\edef\namelist{#1}
        \foreach [count=\counter] \nodename in \namelist {
            \global\edef\numberofnodes{\counter}
            \node at (\nodename) [draw=none,name=hullnode\counter] {};
        }
        \node at (hullnode\numberofnodes) [name=hullnode0,draw=none] {};
        \pgfmathtruncatemacro\lastnumber{\numberofnodes+1}
        \node at (hullnode1) [name=hullnode\lastnumber,draw=none] {};
    },
    create hullnodes
]
($(hullnode1)!#2!-90:(hullnode0)$)
\foreach [
    evaluate=\currentnode as \previousnode using \currentnode-1,
    evaluate=\currentnode as \nextnode using \currentnode+1
    ] \currentnode in {1,...,\numberofnodes} {
-- ($(hullnode\currentnode)!#2!-90:(hullnode\previousnode)$)
  let \p1 = ($(hullnode\currentnode)!#2!-90:(hullnode\previousnode) - (hullnode\currentnode)$),
    \n1 = {atan2(\x1,\y1)},
    \p2 = ($(hullnode\currentnode)!#2!90:(hullnode\nextnode) - (hullnode\currentnode)$),
    \n2 = {atan2(\x2,\y2)},
    \n{delta} = {-Mod(\n1-\n2,360)}
  in 
    {arc [start angle=\n1, delta angle=\n{delta}, radius=#2]}
}
-- cycle
}

\tikzset{
  biggerArrow/.style={
    decoration={markings,mark=at position 1 with {\arrow[scale={1.5}]{>}}},
    postaction={decorate},
    shorten >=0.4pt},
  } 
  
\newcommand{\prd}[1]{\Pi_{#1}}

\newcommand{\sm}[1]{\Sigma \left(#1\right).\,}
\newcommand{\smonlybig}[1]{\Sigma \big(#1\big).\,}

\newcommand{\lam}[1]{\lambda #1 .}
\newcommand{\jdeq}{\equiv}
\newcommand{\defeq}{\vcentcolon\equiv}
\newcommand{\eqvsym}{\simeq}
\newcommand{\eqv}[2]{\ensuremath{#1 \eqvsym #2}\xspace}
\newcommand{\eqvsymspace}{\enspace \eqvsym \enspace}
\newcommand{\eqvspace}[2]{\ensuremath{#1 \eqvsymspace #2}\xspace}
\newcommand{\id}[3][]{\ensuremath{#2 =_{#1} #3}\xspace}
\newcommand{\idtype}[3][]{\ensuremath{\mathsf{Id}_{#1}(#2,#3)}\xspace}
\newcommand{\refl}[1]{\ensuremath{\mathsf{refl}_{#1}}\xspace}
\newcommand{\ct}{
  \mathchoice{\mathbin{\raisebox{0.5ex}{$\displaystyle\centerdot$}}}
             {\mathbin{\raisebox{0.5ex}{$\centerdot$}}}
             {\mathbin{\raisebox{0.25ex}{$\scriptstyle\,\centerdot\,$}}}
             {\mathbin{\raisebox{0.1ex}{$\scriptscriptstyle\,\centerdot\,$}}}
}
\newcommand{\isequiv}{\ensuremath{\mathsf{isequiv}}}
\newcommand{\UU}{\ensuremath{\mathcal{U}}\xspace}
\newcommand{\N}{\ensuremath{\mathbb{N}}\xspace}
\newcommand{\emptyt}{\ensuremath{\mathbf{0}}\xspace}
\newcommand{\unit}{\ensuremath{\mathbf{1}}\xspace}

\newcommand{\bool}{\ensuremath{\mathbf{2}}\xspace}

\newcommand{\iscontr}{\ensuremath{\mathsf{isContr}}}
\newcommand{\isprop}{\ensuremath{\mathsf{isProp}}}
\newcommand{\isset}{\ensuremath{\mathsf{isSet}}}
\def\compare#1#2#3#4{\if#1#3\if#2#41\else0\fi\else0\fi}
\newcommand{\istype}[1]{
  \edef\a{\compare-2#1\empty\empty}
  \if\a1 \iscontr \else
  \edef\b{\compare-1#1\empty\empty}
  \if\b1 \isprop \else
  \edef\c{#1}
  \if0\c \isset \else
  \mathsf{is}\mbox{-}{#1}\mbox{-}\mathsf{type} \fi\fi\fi
}

\newcommand{\Sn}{\mathbb{S}}
\newcommand{\trunc}[2]{\mathopen{}\left\Vert #2\right\Vert_{#1}\mathclose{}}
\newcommand{\tproj}[3][]{\mathopen{}\left|#3\right|_{#2}^{#1}\mathclose{}}
\newcommand{\bproj}[1]{\tproj{}{#1}}
\newcommand{\proptrunc}[1]{\trunc{}{#1}}

\newcommand{\ordinal}[1]{[\mathsf{#1}]} 
\newcommand{\incrmapplain}{\mathrel{\mathrlap{\hspace{2pt}\raisebox{4pt}{\scalebox{0.7}{$+$}}}\mathord{\rightarrow}}}
\newcommand{\incrmap}[2]{\ordinal{#1} \incrmapplain \ordinal{#2}}

\newcommand{\fst}{\mathsf{fst}}
\newcommand{\snd}{\mathsf{snd}}
\newcommand{\choice}[1]{\ensuremath{\mathsf{AC}_{#1}}\xspace}
\newcommand{\infchoice}{\choice \infty}

\newcommand{\const}{\mathsf{const}\xspace}
\newcommand{\cohcond}{\mathsf{coh}}

\newcommand{\implref}[2]{\ensuremath{(\ref{#1}) \! \Rightarrow \! (\ref{#2})}}
\newcommand{\deltop}{\ensuremath{\Delta_+^\mathrm{op}}}
\newcommand{\deltplus}{\ensuremath{\Delta_+}}
\renewcommand{\lim}{\mathsf{lim}\xspace}
\newcommand{\ttfc}{\mathfrak C}
\newcommand{\ccat}{\widehat{\Delta}_+}
\newcommand{\opccat}{\widehat{\Delta}_+^\mathrm{op}}
\newcommand{\cob}[2]{c^{#2}_{\ordinal {#1}}}
\newcommand{\fibslice}[2]{\left( {#1} \slash {#2} \right)_{\mathsf f} }
\newcommand{\proj}{\mathsf{proj}}
\newcommand{\apoint}{\mathfrak {a_0}}
\newcommand{\Aa}[1]{\mathcal T\mkern-6mu{#1}}
\renewcommand{\AA}[1]{\widehat{\mathcal T}\mkern-6mu{#1}}
\newcommand{\Ee}[1]{\mathcal E\mkern-4mu{#1}}
\newcommand{\EE}[1]{\widehat{\mathcal E}\mkern-4mu{#1}}
\newcommand{\Nn}{\mathcal N}
\newcommand{\NN}{\widehat\Nn}
\newcommand{\Pp}{\mathcal P}
\newcommand{\coskel}[3]{\mathsf{coskel}^{\ordinal{#1},{#2}}_{#3}}
\newcommand{\oppo}[1]{#1^\mathrm{op}}
\newcommand{\toomega}{\xrightarrow{\raisebox{-1pt}{$\scriptstyle \omega$}}} 
\newcommand{\ton}[1]{\xrightarrow{\raisebox{-1pt}{$\scriptstyle {\ordinal {#1}}$}}} 

\newcommand{\fib}{\twoheadrightarrow}
\newcommand{\fibre}{\twoheadleftarrow}
\newcommand{\trivfib}{\mathrel{\mathrlap{\hspace{2.5pt}\raisebox{4pt}{$\scriptstyle{\sim}$}}\mathord{\twoheadrightarrow}}}
\newcommand{\trivfibre}{\mathrel{\mathrlap{\hspace{5pt}\raisebox{4pt}{$\scriptstyle{\sim}$}}\mathord{\twoheadleftarrow}}}

\newcommand{\trivcofib}{\mathrel{\mathrlap{\hspace{2.5pt}\raisebox{4pt}{$\scriptstyle{\sim}$}}\mathord{\rightarrowtail}}}
\newcommand{\arrowsim}{\mathrel{\mathrlap{\hspace{2.5pt}\raisebox{4pt}{$\scriptstyle{\sim}$}}\mathord{\rightarrow}}}

\newcommand{\idmorph}[1]{\mathsf{id}_{#1}}
\newcommand{\truncCat}[1]{ \llparenthesis {#1} \rrparenthesis } 
\newcommand{\I}{\mathfrak{I}}
\newcommand{\repr}[1]{\mathsf{Rep}(#1)}
  
\title{The General Universal Property of the Propositional Truncation}

\author{Nicolai Kraus}
\thanks{This work was supported by the Engineering and Physical Sciences Research Council (EPSRC), grant reference EP/M016994/1.}

\begin{document}

\begin{abstract}
In a type-theoretic fibration category in the sense of Shulman (representing a dependent type theory with at least $\unit$, $\Sigma$, $\Pi$, and identity types), we define the type of \emph{coherently constant} functions $A \toomega B$.
This involves an infinite tower of coherence conditions, and we therefore need the category to have Reedy limits of diagrams over $\oppo\omega$.
Our main result is that, if the category further has propositional truncations and satisfies function extensionality, the type of coherently constant function is equivalent to the type $\proptrunc  A \to B$. 

If $B$ is an $n$-type for a given finite $n$, the tower of coherence conditions becomes finite and the requirement of nontrivial Reedy limits vanishes. 
The whole construction can then be carried out in (standard syntactical) homotopy type theory and generalises the universal property of the truncation. 
This provides a way to define functions $\proptrunc  A \to B$ if $B$ is not known to be propositional, and it streamlines the common approach of finding a propositional type $Q$ with $A \to Q$ and $Q \to B$.
\end{abstract}

\maketitle

\section{Introduction}

In homotopy type theory (HoTT), we can \emph{truncate} (\emph{propositionally} or \emph{(-1)-truncate}, to be precise) a type $A$ to get a type $\proptrunc A$ witnessing that $A$ is inhabited without revealing an inhabitant \cite[Chapter 3.7]{HoTTbook}.
This operation roughly corresponds to the \emph{bracket types}~\cite{awodeyBauer_bracketTypes} of extensional Martin-L\"of Type Theory, and to the \emph{squash types}~\cite{nuprl} of NuPRL.

The type $\proptrunc A$ is always propositional, meaning that any two of its inhabitants are equal, and its universal property states that functions $\proptrunc A \to B$ correspond to functions $A \to B$, provided that $B$ is propositional. In particular, we always have a canonical map ${\bproj -}_A : A \to \proptrunc  A$.
This definition is natural and elegant, essentially making the truncation operation a reflector of the subcategory of propositions.
Unfortunately, it can be rather tricky to define a function $\proptrunc A \to B$ if $B$ is not known to be propositional.

One possible way to understand the propositional truncation is to think of elements of $\proptrunc A$ as \emph{anonymous} inhabitants of $A$, with the function ${\bproj -}_A $ hiding the information which concrete element of $A$ one actually has.
With this in mind, let us have a closer look at the mentioned universal property of the propositional truncation, or equivalently, at its elimination principles. 
If we want to find an inhabitant of $\proptrunc A \to B$ and $B$ is propositional, then a function $f : A \to B$ is enough.
A possible interpretation of this fact is that $f$ cannot take different values for different inputs, because $B$ is propositional, justifying that $f$ does (in a certain sense) not have to ``look at'' its argument, such that an anonymous argument is enough.
Note that, we only think of internal properties here. When it comes to computation, the \emph{term} $f$ can certainly behave differently if applied on different \emph{terms} of type $A$.

This thought suggests that, in order to construct and inhabitant of $\proptrunc A \to B$ if $B$ is not necessarily propositional, we need to put a condition on the function $f$ to make sure that it does not distinguish between different inputs.
In other words, we expect that $f$ is required to satisfy some form of \emph{constancy}.
The obvious first try would be to ask for an inhabitant of 
\begin{equation} \label{eq:const-def}
 \const_f \defeq \prd{a^1 a^2 : A} \id{f(a^1)}{f(a^2)}, 
\end{equation}
where we write $=$ for the identity type as it has become standard in HoTT.
The assumption~\eqref{eq:const-def} suffices to derive a function $\proptrunc A \to B$ if we in addition know that $B$ is a set (also called \emph{h-set}, or said to  have \emph{unique identity proofs}). 
As a central concept of HoTT is that the identity type is not always propositional, 
it is not surprising that~\eqref{eq:const-def} 
generally only solves the problem if this additional requirement on $B$ is fulfilled.
If we have a proof that two elements of a type are equal, it will very often matter in which way they are equal. 
Thus, the naive statement that $f$ maps any two points to equal values is usually too weak to construct a map out of the propositional truncation. 
This problem has been studied before~\cite{krausEscardoEtAll_existence,nicolai:thesis}.

Given a function $f:A \to B$ and a proof $c : \const_f$ of weak constancy, we can ask whether the paths (identity proofs) that $c$ gives are well-behaved in the sense that they fit together. Essentially, if we use $c$ to construct two inhabitants of $\id{f(a_1)}{f(a_2)}$, then those inhabitants should be equal.
If we know this, we can weaken the condition that $B$ is a set to the condition that $B$ is a \emph{groupoid} (i.e.\ $1$-truncated), and still construct a function $\proptrunc A \to B$.
This, and the (simpler) case that $B$ is a set as described above, are presented as \cref{satz:special-case-0,satz:special-case-1} in \cref{sec:motivation}. In principle, we could go on and prove the corresponding statement for the case that $B$ is $2$, $3$, \ldots-truncated, each step requiring one additional coherence assumption. 
Unfortunately, handling long sequences of coherence conditions in the direct syntactic way becomes rather unpleasant very quickly.

A setting in which we can deal nicely with such towers of conditions was given by Shulman~\cite{shulman_inversediagrams}, who makes precise the idea that type-theoretic contexts (or ``nested $\Sigma$-types'') correspond to diagrams over inverse categories of a certain shape.
Although we do not require the main result (the construction of univalent models and several applications) of~\cite{shulman_inversediagrams}, we make use of the framework and technical results.
Working in a type-theoretic fibration category in the sense of Shulman, we can further consider the case that this category has Reedy $\oppo\omega$-limits, 
that is, limits of infinite sequences $A_1 \fibre A_2 \fibre A_3 \fibre \ldots$, where every map is a fibration (projection). 
We can think of those limits as ``infinite contexts'' or ``$\Sigma$-types with infinitely many $\Sigma$-components''. 
If these Reedy limits exist, we can formulate the type of \emph{coherently constant} functions from $A$ to $B$, for which we write $A \toomega B$.
We show that such a coherently constant function allows us to define a function $\proptrunc A \to B$, even if $B$ is not known to be $n$-truncated for any finite $n$. 
Even stronger, the type $A \toomega B$ is homotopy equivalent to the type $\proptrunc A \to B$, in the same way as $A \to B$ is equivalent to $\proptrunc A \to B$ under the very strict assumption that $B$ is propositional.

The syntactical version of HoTT as presented in the standard reference~\cite[Appendix A.2]{HoTTbook} does not have (or is at least not expected to have) Reedy $\oppo \omega$-limits.
However, if we consider an $n$-truncated type $B$ for some finite fixed number $n$, then all but finitely many of the coherence conditions captured by $A \toomega B$ become trivial, and that type can be simplified to a finitely nested $\Sigma$-type for which we will write $A \ton {n+1} B$. 
It can be formulated in the syntax of HoTT where we can then prove that, for any $A$ and any $n$-truncated $B$, the type $A \ton {n+1} B$ is equivalent to $\proptrunc  A \to B$.
We thereby generalise the usual universal property of the propositional truncation~\cite[Lemma 7.3.3]{HoTTbook}, because if $B$ is not only $n$-truncated, but propositional, then $A \ton {n+1} B$ can be reduced to $A \to B$ simply by removing contractible $\Sigma$-components.
From the point of view of the standard syntactical version of HoTT, an application of our construction is therefore be the construction of functions $\proptrunc A \to B$ for the case that $B$ is not propositional. 
The usual approach for this problem is to construct a propositional type $Q$ such that $A \to Q$ and $Q \to B$ (see \cite[Chapter 3.9]{HoTTbook}). 
Our construction can be seen as a uniform construction of such a $Q$, since the equivalence $(A \ton {n+1} B) \eqvsym (\proptrunc A \to B)$ 
is proved by constructing a suitable ``contractible extension'' of $A \ton {n+1} B$; the general strategy is to ``expand and contract'' type-theoretic expressions, as we strive to explain with the help of the examples in \cref{sec:motivation}.

Nevertheless, we want to stress that we consider the correspondence between $A \toomega B$ and $\proptrunc A \to B$ in a type-theoretic fibration category with Reedy $\oppo\omega$-limits our main result, and the finite special cases described in the previous paragraph essentially fall out as a corollary.
In fact, we think that Reedy $\oppo \omega$-limits are a somewhat reasonable assumption.
Recently, it has been discussed regularly how these or similar concepts can be introduced into syntactical type theory (for example, see the blog posts by Shulman~\cite{shulman:eating} and Oliveri~\cite{oliveri:formalisedInterpreter} with the comments sections, and the discussion on the HoTT mailinglist~\cite{hott:mailinglist} titled ``Infinitary type theory''). 
Motivations are the question whether HoTT can serve as its own meta-theory, whether we can write an interpreter for HoTT in HoTT, and related questions problems such as the definition of semi-simplicial types~\cite{herbelin_semisimpl}. 
Moreover, a concept that is somewhat similar has been suggested earlier as ``very dependent types''~\cite{Hickey96formalobjects}, even though this suggestion was made in the setting of NuPRL~\cite{nuprl}.

As one anonymous reviewer has pointed out, 
our main result (\cref{thm:result}) can be seen as a type-theoretic, constructive version of Proposition 6.2.3.4 in Lurie's \emph{Higher Topos Theory}~\cite{lurie_topos}. 
This seems to suggest once more that many connections between type theory and homotopy and topos theory are unexplored until now.
The current author has yet to understand the results by Lurie and the precise relationship.

\vspace*{0.2cm}

\subparagraph*{\textbf{Contents.}}
We first discuss the cases that the codomain $B$ is a set or a groupoid, as described in the introduction, in \cref{sec:motivation}. This provides some intuition for our general strategy of proving a correspondence between coherently constant functions and maps out of propositional truncations. In particular, we describe how the method of ``adding and removing contractible $\Sigma$-components'' for proving equivalences can be applied.
In \cref{sec:ttfcs}, we briefly review the notion of a type-theoretic fibration category, of an inverse category, and, most importantly, constructions related to Reedy fibrant diagrams, as described by Shulman~\cite{shulman_inversediagrams}.
Some simple observations about the restriction of diagrams to subsets of the index categories are recorded in \cref{sec:subdiagrams}.
We proceed by defining the \emph{equality diagram} over a given type for a given inverse category in \cref{sec:equality}. The special case where the inverse category is $\deltop$ (the category of nonempty finite sets and strictly increasing functions) gives rise to the \emph{equality semi-simplicial type}, which is discussed in \cref{sec:equality-sst}. We show that the projection of a full n-dimensional tetrahedron to any of its horns is a homotopy equivalence.
Then, in \cref{sec:nat-trans}, we construct a fibrant diagram that represents the exponential of a fibrant and a non-fibrant diagram, with the limit taken at each level.
We extend the category $\deltop$ in \cref{sec:extending-sst}, which allows us to make precise how contractible $\Sigma$-components can be ``added and removed'' in general. Our main result, namely that the types $A \toomega B$ and $\proptrunc A \to B$ are homotopy equivalent, is shown in \cref{sec:maintheorem}.
The finite special cases which can be done without the assumption of Reedy $\oppo\omega$-limits are proved in \cref{sec:infconstancy-finite}, while \cref{sec:conclusions}
is reserved for concluding remarks.

\vspace*{0.2cm}

\subparagraph*{\textbf{Notation.}}
We use type-theoretic notation and we assume familiarity with HoTT, in particular with the book~\cite{HoTTbook} and its terminology.
If $A$ is a type and $B$ depends on $A$, it is standard to write $\prd{a:A}B(a)$ or $\prd A B$ for the type of dependent functions.
For the dependent pair type, we write $\sm{a:A}B(a)$ or $\sm{A}{B}$.
The reason for this apparent mismatch is that we sometimes have to consider nested $\Sigma$-types, and it would seem unreasonable to write all $\Sigma$-components apart from the very last one as subscripts. 
It is sometimes useful to give the last component of a (nested) $\Sigma$-type a name, in which case we allow ourselves to write expressions like $\sm{a:A}\sm{b:B(a)}(c : C(a,b))$.

Regarding notation, one potentially dangerous issue is that there are many different notions of equality-like concepts, such as the identity type of type theory, internal equivalence of types, judgmental equality of type-theoretic expressions, isomorphism of objects in a category, isomorphism or equivalence of categories, and strict equality of morphisms.
For this article, we use the convention that \emph{internal} concepts are written using ``two-line'' symbols, coinciding with the notation of~\cite{HoTTbook}: we write $\id a b$ for the identity type $\idtype a b$, and $A \eqvsym B$ for the type of equivalences between $A$ and $B$.
Other concepts are denoted (if at all) using ``three-line'' symbols: we write $a \jdeq b$ if $a$ and $b$ denote two judgmentally equal expressions, and we use $\jdeq$ for other cases of \emph{strict} equality in the meta-theory. 
By writing $x \cong y$, we express that $x$ and $y$ are isomorphic objects of a category. 
Equality of morphisms (of a category) is sometimes expressed with $\jdeq$, but usually by saying that some diagram commutes, and if we say that some diagram commutes, we always mean that it commutes \emph{strictly}, not only up to homotopy. Other notions of equality are written out.

If $C$ is some category and $x \in C$ an object, we write (as it is standard) $x \slash C$ for the co-slice category of arrows $x \to y$. We do many constructions involving subcategories, but we want to stress that we always and exclusively work with \emph{full} subcategories (apart from the subcategory of fibrations in \cref{def:ttfc}).
Thus, we write $C - x$ for the full subcategory of $C$ that we get by removing the object $x$. Further, if $D$ is a full subcategory of $C$ (we write $D \subset C$) which does not contain $x$, we write $D + x$ for the full subcategory of $C$ that has all the objects of $D$ and the object $x$.

Not exactly notation, but in a similar direction, are the following two remarks: 
First, when we refer to the \emph{distributivity law} of $\Pi$ and $\Sigma$, 
we mean the equivalence
\begin{equation} \label{eq:distributivity}
\prd{a:A}\sm{b:B(a)}C(a,b) \eqvsymspace \sm{f:\prd{a:A}B(a)}\prd{a:A}C(a,f(a))
\end{equation}
which is sometimes called the \emph{type-theoretic axiom of choice} or \emph{$\infchoice$} (see~\cite{HoTTbook}).
Second, if we talk about a \emph{singleton}, we mean a type expression of the form $\sm{a:A}\id a x$ or $\sm{a:A} \id x a$ for a fixed $x$.
The term \emph{singleton} therefore refers to a syntactical shape in which some types can be represented, and it is well-known that those types are contractible.

\section{A First Few Special Cases} \label{sec:motivation}

In this section, we want to discuss some simple examples and aim to build up intuition for the general case.
For now, we work entirely in standard (syntactical) homotopy type theory as specified in~\cite[Appendix A.2]{HoTTbook}, together with function extensionality (see \cite[Appendix A.3.1]{HoTTbook}) and propositional truncation. 
To clarify the latter, we assume that, for any type $A$, there is a propositional type $\proptrunc  A$ with a function ${\bproj -}_A : A \to \proptrunc A$.
Composition with ${\bproj -}_A$ is moreover assumed to induce an equivalence $(\proptrunc A \to B) \eqvsym (A \to B)$.
Due to the ``equivalence reasoning style'' nature of our proofs, we can avoid the necessity of any ``unpleasant manual computation''. 
Thus, we would not benefit from the judgmental computation rule that is usually imposed on the propositional truncation 
(other than not having to assume function extensionality explicitly~\cite{krausEscardoEtAll_existence}).
We think it is worth mentioning that we actually do not require much of the power of homotopy type theory: we only use $1$, $\Sigma$, $\Pi$, identity types, propositional truncations, and assume function extensionality. This will in later sections turn out to be a key feature which enables us to perform the construction in the infinite case (assuming the existence of certain Reedy limits).

Assume we want to construct an inhabitant of $\proptrunc  A \to B$ and $B$ is an $n$-type, for a fixed given $n$. 
The case $n \jdeq -2$ is trivial. For $n \jdeq -1$, the universal property (or the elimination principle) can be applied directly.
In this section, we explain the cases $n \jdeq 0$ and $n \jdeq 1$. 
These results have been formalised in Agda by Danielsson \cite[file \href{http://www.cse.chalmers.se/~nad/listings/equality/H-level.Truncation.html}{\nolinkurl{H-level.Truncation}}]{nisse_library}. 

To begin, we formulate and prove the following auxiliary statement.
Note that the variable $m$ is fixed externally, i.e.\ we do not quantify internally over $m$, as doing so would make it hard to quantify over the type families $C_i$:
\begin{lemma} \label{lem:ignore-trunc-A}
Let $m$ be a fixed natural number. Assume that $A$ is a type and $C_1, C_2, \ldots, C_m$ are type families, where 
$C_j$ may (only) depend on $A$ and on $\prd A C_i$ for all $i < j$, making the expression 
\begin{equation} \label{eq:valid-type}
 \sm{\prd A C_1} \sm{\prd A C_2} \ldots \sm{\prd A C_{m-1}} \prd A C_m
\end{equation}
a well-formed type. 
If we write $C$ for the type \eqref{eq:valid-type}, then the types 
$C$ and $\proptrunc A \to C$ are equivalent.
\end{lemma}
\begin{proof}[Proof of \cref{lem:ignore-trunc-A}]
 This holds essentially by the usual distributivity law~\eqref{eq:distributivity} of $\Pi$ (or $\to$) and $\Sigma$, together with the equivalence $\proptrunc  A \times A \eqvsym A$.
 In detail, for a type $A$, a type family $B$ indexed over $A$, and a second family $D$ indexed over $\prd A B$, we have the following chain of equivalences:
\begin{alignat*}{3}
  &&& \proptrunc A \to \left(\sm{g : \prd{A} B} D(g)\right) \\
  \intertext{\hspace{1.5cm}(by distributivity)} 
   \quad & \eqvsym & \quad & \sm{g : \proptrunc A \to \prd{A} B} \prd{x : \proptrunc A} D(g(x))     \\
  \intertext{\hspace{1.5cm}(as $\prd{A}B$ and $\proptrunc A \to \prd{A} B$ are equivalent via $g \mapsto \lam y g$)}
  & \eqvsym && \sm{g : \prd{A}B} \prd{x:\proptrunc{A}}D(\lam y g(x))  \\
  \intertext{\hspace{1.5cm}(as $x = y$ for $x,y : \proptrunc A$, thus $g(x) = g(y)$)}
  & \eqvsym && \sm{g : \prd{A}B} \prd{x:\proptrunc{A}}D(\lam y g(y))  \\
  & \eqvsym && \sm{g : \prd{A}B} \left(\proptrunc{A} \to D(g)\right)  
\end{alignat*}
We do induction on $m$. The case $m \jdeq 1$ (as well as the degenerated case $m \jdeq 0$) is immediate.
For $m \geq 2$, the type $C$ as considered in the lemma can be written as $\sm{\prd A B}D$, and by the above construction, $\proptrunc A \to C$ it is thus equivalent to 
$\sm{g : \prd A B} \proptrunc A \to D(g)$.  
\end{proof}

\subsection{Constant Functions into Sets}

We consider the case $n \jdeq 0$ first; that is, we assume that $B$ is a set.
Recall the definition of $\const$ given in \eqref{eq:const-def}.

\begin{satz}[case $n \jdeq 0$] \label{satz:special-case-0}
 Let $B$ be a set and $A$ be any type. Then, we have the equivalence
 \begin{equation} \label{eq:special-case-0}
 \eqvspace{(\proptrunc  A \to B)}{\sm{f : A \to B} \const_f}.
 \end{equation}
\end{satz}
Note that, if $B$ is not only a set but even a propositional type, the condition $\const_f$ is not only automatically satisfied, but it is actually contractible as a type. By the usual equivalence lemmata, the type on the right-hand side of \eqref{eq:special-case-0} then simplifies to $(A \to B)$, which exactly is the universal property.
Thus, we view \eqref{eq:special-case-0} as a first generalisation.

\begin{proof}[Proof of \cref{satz:special-case-0}]
 Assume $\apoint : A$ is some point in $A$.
 In the following, we construct a chain of equivalences. The variable names for certain $\Sigma$-components might seem somewhat odd: for example, we introduce a point $f_1 : B$. 
 The reason for this choice will become clear later. 
 For now, we simply emphasise that $f_1$ is ``on the same level'' as $f : A \to B$ in the sense that they both give points, rather than for example paths (like, for example, an inhabitant of $\const_f$).
\begin{equation}
 \begin{alignedat}{4} 
 && &&& B \\
  \text{(S1)} && \quad & \eqvsym & \quad & \sm{f_1 : B} \big(A \to \sm{b : B} \id b {f_1} \big)  \\
  \text{(S2)} &&& \eqvsym && \sm{f_1 : B} \sm{f : A \to B} \prd {a:A} \id{f(a)}{f_1} \\
  \text{(S3)} &&& \eqvsym && \sm{f_1 : B} \sm{f : A \to B} (\prd {a:A} \id{f(a)}{f_1}) \times (\const_f) \times (\id{f(\apoint)}{f_1}) \\
  \text{(S4)} &&& \eqvsym && \sm{f : A \to B} (\const_f) \times \sm{f_1 : B} (\id{f(\apoint)}{f_1}) \times (\prd {a:A} \id{f(a)}{f_1}) \\
  \text{(S5)} &&& \eqvsym && \sm{f : A \to B} (\const_f) \times \left(\sm{f_1 : B} \id{f(\apoint)}{f_1}\right) \\
  \text{(S6)} &&& \eqvsym && \sm{f : A \to B} \const_f 
 \end{alignedat}
\end{equation}
 Let us explain the validity of the single steps. In the first step, we add a family of singletons. In the second step, we apply the distributivity law \eqref{eq:distributivity}. In the third step, we add two $\Sigma$-components, and $B$ being a set ensures that both of them are propositional. 
 But it is very easy to derive both of them from $\prd{a:A}\id{f(a)}{f_1}$, showing that both of them are contractible.
 In the fourth step, we simply reorder some $\Sigma$-components, and in the fifth step, we use that $\prd{a:A}\id{f(a)}{f_1}$ is contractible by an argument analogous to that of the third step. Finally, we can remove two $\Sigma$-components which form a contractible singleton.
 
 If we carefully trace the equivalences, we see that the function part 
 \begin{equation}
  e : B \; \to \; \sm{f : A \to B} \const_f
 \end{equation}
 is given by
 \begin{equation}
  e(b) \jdeq \left(\lam a b \, , \, \lam {a^1a^2} \refl b \right),
 \end{equation}
 not depending on the assumed $\apoint:A$. But as $e$ is an equivalence assuming $A$, it is also an equivalence assuming $\proptrunc  A$. 
 
 As $\proptrunc  A \to \big( B \eqvsym \left(\sm{f : A \to B} \const_f \right)\big)$ implies that the two types $(\proptrunc  A \to B)$ and $\big(\proptrunc  A \to \left(\sm{f : A \to B} \const_f \right)\big)$ are equivalent, the statement follows from \cref{lem:ignore-trunc-A}.
\end{proof}

The core strategy of the steps (S1) to (S6) is to add and remove contractible $\Sigma$-components, and to reorder and regroup them. This principle of expanding and contracting a type expression can be generalised and, as we will see, even works for the infinite case when $B$ is not known to be of any finite truncation level. 
Generally speaking, we use two ways of showing that components of $\Sigma$-types are contractible. The first is to group two of them together such that they form a singleton, as we did in (S1) and (S6). 
The second is to use the fact that $B$ is truncated, as we did in (S3). 
We consider the first to be the key technique, and in the general (infinite) case of an untruncated $B$, the second can not be applied at all.
We thus view the second method as a tool to deal with single $\Sigma$-components that lack a ``partner'' only because the case that we consider is finite, and which is unneeded in the infinite case.

\subsection{Constant Functions into Groupoids}

The next special case is $n \jdeq 1$. Assume that $B$ is a $1$-type (sometimes called a \emph{groupoid}).
Let us first clarify which kind of constancy we expect for a map $f : A \to B$ to be necessary.
Not only do we require $c : \const_f$, we also want this constancy proof (which is in general not propositional any more) to be \emph{coherent}:
given $a^1$ and $a^2 : A$, we expect that $c$ only allows us to construct essentially \emph{one} proof of $\id{f(a^1)}{f(a^2)}$.
The reason is that we want the data (which includes $f$ and $c$) together to be just as powerful as a map $\proptrunc  A \to B$, and from such a map, we only get trivial loops in $B$.

We claim that the required coherence condition is
\begin{equation} \label{eq:coh-cond}
 \cohcond_{f,c} \defeq \prd{a^1 a^2 a^3 : A} c(a^1,a^2) \cdot c(a^2,a^3) = c(a^1,a^3).
\end{equation}
A first sanity check is to see whether from $d: \cohcond_{f,c}$ we can now prove that $c(a,a)$ is equal to $\refl{a}$, something that should definitely be the case if we do not want to be able to construct possibly different parallel paths in $B$.
To give a positive answer, we only need to see what $d(a,a,a)$ tells us.

\begin{satz}[case $n \jdeq 1$] \label{satz:special-case-1}
 Let $B$ be a groupoid ($1$-type) and $A$ be any type. Then, we have 
 \begin{equation}
  \eqvspace {(\proptrunc  A \to B)} {\big( \sm{f:A \to B} \sm{c : \const_f} \cohcond_{f,c} \big)}.
 \end{equation}
\end{satz}

Note that \cref{satz:special-case-1} generalises \cref{satz:special-case-0}: if $B$ is a set (as in \cref{satz:special-case-0}), it is also a groupoid and the type $\cohcond_{f,c}$ becomes contractible, as it talks about equality of equalities.

\begin{proof}
 Although not conceptually harder, it is already significantly more tedious to write down the chain of equivalences.
 We therefore choose a slightly different representation. Assume $\apoint : A$ as before. We then have:
\begin{equation}
 \begin{alignedat}{4} 
 && &&& \phantom{\Sigma (}  B \\
  \text{(S1)} && \quad & \eqvsym & \quad & \\
  \text{} && \quad &  & \quad & \sm{f_1 : B}   \\ 
  \text{} && \quad &  & \quad & \sm{f : A \to B} \sm{c_1 : \prd{a:A} \id{f(a)}{f_1}} \\
  \text{} && \quad &  & \quad & \sm{c : \const_f} \sm{d_1 : \prd{a^1 a^2:A} \id{c(a^1,a^2) \cdot c_1(a^2)}{c_1(a^1)}} \\
  \text{} && \quad &  & \quad & \sm{c_2 : \id{f(\apoint)}{f_1}} \sm{d_3 : \id{c(\apoint,\apoint) \cdot c_1(\apoint)}{c_2}} \\
  \text{} && \quad &  & \quad & \sm{d : \cohcond_{f,c}} \\
  \text{} && \quad &  & \quad & \phantom{\Sigma }  \left(d_2 : \prd{a:A} c(\apoint,a) \cdot c_1(a) = c_2\right) \\
  \text{(S2)} && \quad & \eqvsym & \quad & \\
  \text{} && \quad &  & \quad & \sm{f : A \to B} \sm{c : \const_f} \sm{d : \cohcond_{f,c}} \\
  \text{} && \quad &  & \quad & \sm{f_1 : B} \sm{c_2 : \id{f(\apoint)}{f_1}} \\
  \text{} && \quad &  & \quad & \sm{c_1 : \prd{a:A} \id{f(a)}{f_1}} \sm{d_2 : \prd{a:A} c(\apoint,a) \cdot c_1(a) = c_2} \\
  \text{} && \quad &  & \quad & \sm{d_1 : \prd{a^1 a^2:A} \id{c(a^1,a^2) \cdot c_1(a^2)}{c_1(a^1)}} \\
  \text{} && \quad &  & \quad & \phantom{\Sigma }  \left(d_3 : \id{c(\apoint,\apoint) \cdot c_1(\apoint)}{c_2}\right) \\
  \text{(S3)} && \quad & \eqvsym & \quad & \\
  \text{} && \quad &  & \quad & \sm{f : A \to B} \sm{c : \const_f} (d : \cohcond_{f,c}) 
 \end{alignedat}
\end{equation}
 In the first step (S1), we 
 expand the single type $B$ to a nested $\Sigma$-type with in total nine $\Sigma$-components.
 We write them in six lines, and each line apart from the first is a contractible part of this nested $\Sigma$-type, implying that the whole type is equivalent to $B$.
 In the lines two and three, we can apply the distributivity law, i.e.\ the equivalence~\eqref{eq:distributivity}, to give them the shape of singletons, while the fourth line is already a singleton.
 As $B$ is $1$-truncated, the lines five and six represent propositional types, but those types are easily seen to be inhabited using the other $\Sigma$-components. 

 In the second step, we simply re-order some $\Sigma$-components. Then, in step (S3), we remove the $\Sigma$-components in the lines two to five which is justified as, again, each line represents a contractible part of the nested $\Sigma$-type.
 
 We trace the canonical equivalences to see that the function-part of the constructed equivalence is
 \begin{align}
  &e : B \to \sm{f : A \to B} \sm{c : \const_f} (d : \cohcond_{f,c}) \\
  &e(b) \jdeq (\lam a b \, , \, \lam {a^1 a^2} \refl{b} \, , \, \lam {a^1 a^2 a^3} \refl{\refl b}).
 \end{align}
 In particular, $e$ is independent from the assumed $\apoint : A$. 
 As before, this means that $e$ is an equivalence assuming $\proptrunc  A$, and, with the help of \cref{lem:ignore-trunc-A}, we derive the claimed equivalence.
\end{proof}

\subsection{Outline of the General Idea}  \label{subsec:outline-of-inf}

At this point, it seems plausible that what we have done for the special cases of $n \jdeq 0$ and $n \jdeq 1$ can be done for any (fixed) $n < \infty$. 
Nevertheless, we have seen that the case of groupoids is already significantly more involved than the case of sets.
To prove a generalisation, we have to be able to state what it means for a function to be ``coherently constant'' on $n$ levels, rather than just the first one or two.

Let us try to specify what ``coherently constant'' should mean in general.
If we have a function $f : A \to B$, we get a point in $B$ for any $a:A$.
A constancy proof $c : \const_f$ gives us, for any pair of points in $A$, a path between the corresponding points in $B$.
Given three points, $c$ gives us three paths which form a ``triangle'', and an inhabitant of $\cohcond_{f,c}$ does nothing else than providing a filler for such a triangle.
It does not take much imagination to assume that, on the next level, the appropriate coherence condition should state that the ``boundary'' of a tetrahedron, consisting of four filled triangles, can be filled.

To gain some intuition, let us look at the following diagram:
\begin{figure}[H]
\begin{center}
\begin{tikzpicture}[align=left, node distance=2.5cm]

  \node [color=black](A1) {$A$}; 
  \node [above of=A1](A2) {$A \times A$}; 
  \node [above of=A2](A3) {$A \times A \times A$}; 
  
  \node [right of=A1, node distance=\textwidth/1.6] (B1) {$B$}; 
  \node [above of=B1] (B2) {$\sm{b_1,b_2:B} \id{b_1}{b_2}$}; 
  \node [above of=B2] (B3) {$\sm{b_1,b_2,b_3 : B}$ \\ $\sm{p_{12}: \id{b_1}{b_2}}$ \\ $\sm{p_{23}: \id{b_2}{b_3}}$ \\ $\sm{p_{13}: \id{b_1}{b_3}}$ \\ \phantom{$\Sigma$} $p_{12}\cdot p_{23} = p_{13}$}; 
  
  \draw[->, transform canvas={xshift=-0.25ex}] (A2) to node {} (A1);
  \draw[->, transform canvas={xshift=0.25ex}] (A2) to node {} (A1);

  \draw[->, transform canvas={xshift=-0.5ex}] (A3) to node {} (A2);
  \draw[->, transform canvas={xshift=0ex}] (A3) to node {} (A2);
  \draw[->, transform canvas={xshift=0.5ex}] (A3) to node {} (A2);
  
  \draw[->, transform canvas={xshift=-0.25ex}] (B2) to node {} (B1);
  \draw[->, transform canvas={xshift=0.25ex}] (B2) to node {} (B1);

  \draw[->, transform canvas={xshift=-0.5ex}] (B3) to node {} (B2);
  \draw[->, transform canvas={xshift=0ex}] (B3) to node {} (B2);
  \draw[->, transform canvas={xshift=0.5ex}] (B3) to node {} (B2);
  
  \node [right of=A1, node distance=2.5cm] (A1a) {}; 
  \node [right of=A2, node distance=2.5cm] (A2a) {}; 
  \node [right of=A3, node distance=2.5cm] (A3a) {}; 
  \node [left of=B1, node distance=2.5cm] (B1b) {}; 
  \node [left of=B2, node distance=2.5cm] (B2b) {}; 
  \node [left of=B3, node distance=2.5cm] (B3b) {};

  \draw[->, dashed] (A1a) to node [above] {$t_{\ordinal{0}}$} (B1b);
  \draw[->, dashed] (A2a) to node [above] {$t_{\ordinal{1}} : \const_{t_{\ordinal{0}}}$} (B2b);
  \draw[->, dashed] (A3a) to node [above] {$\cohcond_{t_{\ordinal{0}}, t_{\ordinal{1}}}$} (B3b);
\end{tikzpicture}
\caption{Constancy as a natural transformation} \label{fig:const-as-nat-trans}
\end{center}
\end{figure}
All vertical arrows are given by projections. 
Consider the category $D$ with objects the finite ordinals $\ordinal 0$, $\ordinal 1$ and $\ordinal 2$ (with $1$, $2$, and $3$ objects, respectively), and arrows the strictly monotonous maps. Then, the left-hand side and the right-hand side can both be seen as a diagram over $\oppo D$. 
The data that we need for a ``coherently constant function'' from $A$ into $B$, if $B$ is a groupoid, can now be viewed as a natural transformation $t$ from the left to the right diagram (the dashed horizontal arrows). On the lowest level, such a natural transformation consists of a function $t_{\ordinal{0}} : A \to B$, which we called $f$ before. On the next level, 
we have $t_{\ordinal{1}} : A^2 \to \sm{b_1,b_2:B}b_1=b_2$, but in such a way that the diagram commutes (strictly, not up to homotopy), enforcing
\begin{equation}
 \fst(t_{\ordinal{1}}(a^1,a^2)) \jdeq (t_{\ordinal{0}}(a^1), t_{\ordinal{0}}(a^2))
\end{equation}
and thereby making $t_{\ordinal{1}}$ the condition that $t_{\ordinal{0}}$ is weakly constant. Finally, $t_{\ordinal{2}}$ yields the coherence condition $\cohcond$.

In the most general case, where we do not put any restriction on $B$, we certainly cannot expect that a finite number of coherence conditions can suffice.
Instead of the diagram over $\oppo D$, as pictured on the right-hand side of \cref{fig:const-as-nat-trans}, 
we will need a diagram over the the category of all non-zero finite ordinals.
This is what we call the \emph{equality semi-simplicial type} over $B$, written $\Ee{B}$. 
In the language of model categories, this is a fibrant replacement of the constant diagram.
It would be reasonable to expect that our $\Ee{B}$ extends the diagram shown in \cref{fig:const-as-nat-trans}, but this will only be true up to (levelwise) equivalence of types.
Defining $\Ee{B}$ as a strict extension of that diagram is tempting, but it seems to be combinatorically nontrivial to continue in the same style, as it would basically need Street's orientals~\cite{Street1987283}.
Our construction will be much simpler to write down and easier to work with, with the only potential disadvantage being that, compared to the diagram \cref{fig:const-as-nat-trans}, the lower levels will look rather bloated.
The other diagram in \cref{fig:const-as-nat-trans}, i.e.\ the left-hand side, is easy to extend, and we call it the \emph{trivial} diagram over $A$.
In the terminology of simplicial sets, it is the $\ordinal 0$-coskeleton of the constant diagram. 
Our main result is essentially an internalised version, stated as an equivalence of types, of the following slogan:
\begin{center}
\begin{minipage}[]{\textwidth *4/5} 
\emph{Functions $\proptrunc A \to B$ correspond to natural transformations from the trivial diagram over $A$ to the semi-simplicial equality type over $B$.}
\end{minipage} 
\end{center}
Our type of natural transformations is basically a Reedy limit of an exponential of diagrams. 
We will perform the \emph{expanding and contracting} principle that we have exemplified in the proofs of \cref{satz:special-case-0,satz:special-case-1} by modifying the index category of the diagram of which we take the limit, step by step, taking care that every single step preservers the Reedy limit in question up to homotopy equivalence.
As we will see, these steps correspond indeed to the steps that we took in the proofs of \cref{satz:special-case-0,satz:special-case-1}.

\section{Fibration Categories, Inverse Diagrams, and Reedy Limits} \label{sec:ttfcs}

In his work on \emph{Univalence for Inverse Diagrams and Homotopy Canonicity}, Shulman has proved several deep results~\cite{shulman_inversediagrams}.
Among other things, he shows that diagrams over inverse categories can be used to build new models of univalent type theory, and uses this to prove a partial solution to Voevodsky's homotopy-canonicity conjecture. 
We do not require those main results; in fact, we do not even assume that there is a universe, and consequently we also do not use univalence! 
At the same time, what we want to do can be explained nicely in terms of diagrams over inverse diagrams, and we therefore choose to work in the same setting. 
Luckily, it is possible to do this with only a very short introduction to type-theoretic fibration categories, inverse diagrams and Reedy limits,
and this is what the current section servers for.

\vspace*{0.2cm}

\subparagraph{\textbf{Type-theoretic fibration categories.}}
A \emph{type-theoretic fibration category} (as defined in~\cite[Definition~2.1]{shulman_inversediagrams} is a category with some structure that allows to model dependent type theory with identity types.
Let us recall the definition, where we use a lemma by Shulman to give an equivalent (more ``type-theoretic'') formulation:

\begin{definition}[{Type-theoretic fibration category,~\cite[Definition 2.1 combined with Lemma 2.4]{shulman_inversediagrams}}] \label{def:ttfc}
 A type-theoretic fibration category is a category $\ttfc$ which has the following structure.
 \begin{enumerate}[label={(\roman*)},ref={\roman*}]
  \item A terminal object $\unit$.
  \item A (not necessarily full) subcategory $\mathfrak{F} \subset \ttfc$ containing all the objects, all the isomorphisms, and all the morphisms with codomain $\unit$. A morphism in $\mathfrak F$ is called a \emph{fibration}, and written as $A \fib B$.
  Any morphism $i$ is called an \emph{acyclic cofibration} and written $i : X \trivcofib Y$ if it has the left lifting property with respect to all fibrations, meaning that every commutative square
  \begin{center}
  \begin{tikzpicture}[align=left, node distance=1.2cm]
  \node [](XX) {$X$}; 
  \node [below of=XX](XY) {$Y$}; 
  \node [right of=XX](XA) {$A$}; 
  \node [below of=XA](XB) {$B$}; 
  
  \draw[>->] (XX) to node {$i$\;\; $\sim$} (XY); 
  \draw[->] (XX) to node {} (XA);
  \draw[->>] (XA) to node [right] {$f$} (XB);
  \draw[->] (XY) to node {} (XB);
  \end{tikzpicture}
  \end{center}
  has a (not necessarily unique) filler $h:Y \to A$ that makes both triangles commute.
  \item All pullbacks of fibrations exist and are fibrations.
  \item For every fibration $g:A \fib B$, the pullback functor $g^\star : \ttfc\slash B \to \ttfc \slash A$ has a partial right adjoint $\Pi_g$, defined at all fibrations over $A$, whose values are fibrations over $B$. 
  \item For any fibration $A \fib B$, the diagonal morphism $A \to A \times_B A$ factors as $A \trivcofib P_BA \fib A \times_B A$, with the first map being an acyclic cofibration and the second being a fibration. \label{item:ttfc-def-factor}
  \item For any $A \fib B$, there exists a factorisation as in \eqref{item:ttfc-def-factor} such that in any diagram of the shape
  \begin{center}
  \begin{tikzpicture}[align=left, node distance=1.2cm]
  \node [](X) {$X$}; 
  \node [right of=X](Y) {$Y$}; 
  \node [right of=Y](Z) {$Z$}; 
  \node [below of=X](A) {$A$}; 
  \node [below of=Y](B) {$P_B A$}; 
  \node [below of=Z](C) {$B$}; 
  
  \draw[>->] (A) to node [above] {$\sim$} (B);
  \draw[->>] (B) to node {} (C);
  \draw[->] (X) to node {} (Y);
  \draw[->] (Y) to node {} (Z);
  \draw[->] (X) to node {} (A);
  \draw[->] (Y) to node {} (B);
  \draw[->] (Z) to node {} (C);
  \end{tikzpicture}
  \end{center}
  we have the following: if both squares are pullback squares (which implies that $Y \to Z$ and $X \to Z$ are fibrations), then $X \to Y$ is an acyclic cofibration.
 \end{enumerate}
\end{definition}
\begin{remark}
 From the above definition, it follows that \emph{every} morphism factors as an acyclic cofibration followed by a fibration.
 Shulman's proof~\cite[Lemma 2.4]{shulman_inversediagrams}, a translation of the proof by Gambino and Garner~\cite{gambinoGarner_ITwfs} into category theory, relies on the fact that every morphism $A \to \unit$ is a fibration (``all objects are fibrant'') by definition.
\end{remark}

The example of a type-theoretic fibration category that we mainly have in mind is~\cite[Example 2.9]{shulman_inversediagrams}, the category of contexts of a dependent type theory with a unit type, $\Sigma$- and $\Pi$-types, and identity types.
The unit, $\Sigma$- and $\Pi$-types are required to satisfy judgmental $\eta$-rules.
Because of these $\eta$-rules, we do not need to talk about contexts; we can view every object of the category as a nested $\Sigma$-type with some finite number of components. Of course, the terminal object is the unit type. The subset of fibrations is the closure of the projections under isomorphisms. One nice property is that the $\eta$-rules also imply that we can assume that all fibrations are a projection of the form $\big(\sm{x:X}Y(x)\big) \fib X$.
Pullbacks correspond to substitutions, and the partial functor $\Pi_g$ comes from dependent function types.
For any fibration $f : A \fib B$, the factorisation in item \eqref{item:ttfc-def-factor} can be obtained using the intensional identity type: if $B$ is the unit type, then the factorisation can be written as $A \trivcofib \big(\sm{(x,y):A\times A}\id{x}{y}\big) \fib A \times A$, and similar otherwise (see~\cite{gambinoGarner_ITwfs}). The acyclic cofibration is given by reflexivity. 

Note that the type theory specified in the standard reference on HoTT~\cite[Appendix A.2]{HoTTbook}
does not have judgmental $\eta$-rules for $\Sigma$ and $\unit$.
This does not constitute a problem when we want to apply our results to homotopy type theory.
First, it appears to be an arbitrary choice of~\cite{HoTTbook} to not include these judgmental $\eta$-rules in the theory.
There does not seem to be any fundamental difficulty with them, and the implementations Agda and Coq do indeed support them.
Second, as Shulman states, these judgmental $\eta$-rules are convenient but not really necessary~\cite[Example 2.9]{shulman_inversediagrams}.
This is certainly true for our constructions that we can do with finitely nested $\Sigma$-types, although it is likely that the assumption of $\oppo\omega$-limits (infinitely nested $\Sigma$-types) would have to be phrased more carefully in the absence of judgmental $\eta$-conversions (see our proof of \cref{thm:result}).

Given a type-theoretic fibration category $\ttfc$ with an object $A$, we can think of $A$ as a context.
Type theoretically, we can work in the theory over the fixed context $A$.
Categorically, this means we work in the slice over $A$.
The slice category $\ttfc \slash A$ is not necessarily a type-theoretic fibration category as not all morphisms $B \to A$ are fibrations, but we can simply restrict ourselves to those that are.
Shulman denotes this full subcategory of $\ttfc \slash A$ by $\fibslice \ttfc A$.
The observation that the (restricted) slice of a type-theoretic fibration category is again a type-theoretic fibration category allows us that, when we want to do an ``internally expressible'' construction for any general given fibration, we can without loss of generality assume that the codomain of the fibration is the unit type.
This corresponds to the fact that an ``internal'' construction in type theory still works if we add additional assumptions to the context (which are then simply ignored by the construction).

It is not exactly true that a type-theoretic fibration category has an intensional dependent type theory as its internal language due to the well-known issue that substitution in type theory is strictly functorial. 
Fortunately, coherence theorems (see e.g.~\cite{awodey_natural,lumsdaineWarren_overlooked}) can be applied to solve this problem, and we do not worry about it but simply refer to Shulman's explanation \cite[Chapter 4]{shulman_inversediagrams}. 
The crux is that, disregarding these coherence issues, the syntactic category of the dependent type theory with $\unit$, $\Sigma$, $\Pi$, and identity types is essentially the initial type-theoretic fibration category. 
A consequence we will exploit heavily is that, when reasoning about type-theoretic fibration categories, we can use type-theoretic constructions freely as long as they can be performed using $\unit$, $\Pi$, $\Sigma$, and identity types. 
For example, the same notion of function extensionality and type equivalence $A \eqvsym B$ can be defined.
This means, of course, that we have to be very careful with the terminology. 
We call a morphism that is an equivalence in the type-theoretic sense a \emph{homotopy equivalence}, written $A \arrowsim B$, while an \emph{isomorphism} is really an isomorphism in the usual categorical sense. 
Note that any isomorphism is not only a fibration by definition, but it is automatically an acyclic cofibration, and acyclic cofibrations are further automatically homotopy equivalences.
Further, it is natural to introduce the following terminology:
\begin{definition}[acyclic fibration]
 We say that a morphism is an \emph{acyclic fibration} if it is a fibration and a homotopy equivalence.
\end{definition}
An important property to record is that acyclic fibrations are stable under pullback~\cite[Corollary 3.12]{shulman_inversediagrams}.
In diagrams, we write $A \trivfib B$ for acyclic fibrations.

\vspace*{0.2cm}

\subparagraph{\textbf{Inverse categories and Reedy fibrant diagrams.}}
For objects $x$ and $y$ of a category, write $y \prec x$ if $y$ receives a nonidentity morphism from $x$ (and $y \preceq x$ if $y \prec x$ or $y \jdeq x$).
A category $\I$ is called an \emph{inverse category} (also sometimes called \emph{one-way category}) if the relation $\prec$ is well-founded.
In this case, the \emph{ordinal rank} of an object $x$ in $\I$ is defined by 
\begin{equation}
 \rho(x) \defeq \sup_{y \prec x}(\rho(y) + 1).
\end{equation}
As described by Shulman \cite[Section 11]{shulman_inversediagrams}, diagrams on $\I$ can be constructed by well-founded induction in the following way.
If $x$ is an object, write $x \sslash \I$ for the full subcategory of the co-slice category $x \slash \I$ which excludes only the identity morphism $\idmorph x$.
Consider the full subcategory $\Set{y | y \prec x} \subset \I$. There is the forgetful functor $U : x \sslash \I \to \Set{y | y \prec x}$, mapping any $x \xrightarrow f y$ to its codomain $y$.
If further $A$ is a diagram in a type-theoretic fibration category $\ttfc$ that is defined on this full subcategory, if the limit 
\begin{equation}
 M^A_x \defeq \lim_{x \sslash \I} (A \circ U).
\end{equation}
exists, it it called the corresponding \emph{matching object}.
To extend the diagram $A$ to the full subcategory $\Set{y | y \preceq x} \subset \I$, it is then sufficient to give an object $A_x$ and a morphism $A_x \to M^A_x$.
The diagram $A : \I \to \ttfc$ is \emph{Reedy fibrant} if all matching objects $M^A_x$ exist and all the maps $A_x \to M^A_x$ are fibrations.
We use the fact that fibrations can be regarded as ``one-type projections'' in the following way:

\begin{definition}[Decomposition in matching object and fibre] \label{notation:match-and-fibre}
 If $A : \I \to \ttfc$ is a Reedy fibrant diagram, we write (as said above) $M^A_x$ for its matching objects, and $F^A(x,m)$ for the fibre over $m$; that is, we have
 \begin{equation}
  A_x \; \cong \; \sm{m:M^A_x}F^A(x,m).
 \end{equation}
\end{definition}

There is the more general notion of a \emph{Reedy fibration} (a natural transformation between two diagrams over $\I$ with certain properties), so that a diagram is Reedy fibrant if and only if the unique transformation to the terminal diagram is a Reedy fibration.
Further, $\ttfc$ is said to have \emph{Reedy $\I$-limits} if any Reedy fibrant $A : \I \to \ttfc$ has a limit which behaves in the way one would expect; in particular, if a natural transformation between two Reedy fibrant diagrams is levelwise a homotopy equivalence, then the map between the limits is a homotopy equivalence.
We omit the exact definitions as our constructions do not require them and refer to \cite[Chapter 11]{shulman_inversediagrams} for the details instead.
For us, it is sufficient to record that a consequence of the definition of having Reedy $\oppo\omega$-limits is the following:
\begin{lemma} \label{lem:acyclic_limit}
 Let a type-theoretic fibration category $\ttfc$ that has Reedy $\oppo\omega$-limits be given.
 Suppose that
 \begin{equation}
  F \defeq F_{\mathsf 0} \trivfibre F_{\mathsf 1} \trivfibre F_{\mathsf 2} \trivfibre \ldots
 \end{equation}
 is a diagram $F : \oppo\omega \to \ttfc$, where all maps are acyclic fibrations. 
 For each $i$, the canonical map $\lim (F) \to F_i$ is a homotopy equivalence.
\end{lemma}
\begin{proof}
 Consider the diagram that is constantly $F_i$ apart from a finite part,
 \begin{equation}
  G := F_{\mathsf 0} \trivfibre F_{\mathsf 1} \trivfibre \ldots \trivfibre F_{i-1} \trivfibre F_{i} \trivfibre F_{i} \trivfibre F_{i} \ldots .
 \end{equation}
 There is a canonical natural transformation $F \to G$, induced by the arrows in $F$, which is a Reedy fibration and levelwise an acyclic fibration. 
 It follows directly from the precise definition of Reedy limits \cite[Definition 11.4]{shulman_inversediagrams} that the induced map between the limits $\lim(F) \to F_i$ is a fibration and a homotopy equivalence.
\end{proof}

For later, we further record the following two simple lemmata:
\begin{lemma} 
 If $A : \I \to \ttfc$ is Reedy fibrant, then so is $A \circ U : x \slash \I \to \ttfc$.
\end{lemma}
\begin{proof}
 This is due to the fact that for a (nonidentity) morphism $k : x \to y$ in $\I$ the categories $k \sslash (x \sslash \I)$ and $y \sslash \I$ are isomorphic.
 This argument is already used by Shulman (\cite[Lemma 11.8]{shulman_inversediagrams}).
\end{proof}
\begin{lemma} \label{lem:change-notation-poset}
 If $\I$ is 
 a poset (a partially ordered set), $x$ an object, $A : \I \to \ttfc$ a diagram, and the limit $\lim_{x \sslash \I} (A \circ U)$ exists, then $\lim_{\Set{y | y \prec x}} A$ exists as well and both are isomorphic. \qed
\end{lemma}

An inverse category $\I$ is \emph{admissible} for $\ttfc$ if $\ttfc$ has all Reedy $(x\sslash \I)$-limits.
If $\I$ is finite, then any type-theoretic fibration category has Reedy $\I$-limits by \cite[Lemma 11.8]{shulman_inversediagrams}.
From the same lemma, it follows that for all constructions that we are going to do, it will be sufficient if $\ttfc$ has Reedy $\oppo\omega$-limits.
Further, in all our cases of interest, all co-slices of $\I$ are finite, and $\ttfc$ is automatically admissible.

Because of the above, let us fix the following:
\begin{notation}
 For the rest of this article, let $\ttfc$ be a type-theoretic fibration category with Reedy $\oppo\omega$-limits, which further satisfies function extensionality.
 We refer to the objects of $\ttfc$ (which are by definition always fibrant) as \emph{types}.
 Let us further introduce the term \emph{tame category}.
 We say that an inverse category is a \emph{tame category} if all co-slices $x \slash \I$ are finite (which implies that $\rho(x)$ is finite for all objects $x$) and, for all $n$, the set of objects at ``level'' $n$, that is $\Set{x \in \I | \rho(x) \jdeq n}$, is finite. The important property is that a tame category $\I$ is admissible for $\ttfc$, and that $\ttfc$ has Reedy $\I$-limits.
 Thus, tame categories make it possible to perform constructions without worrying whether required limits exist, and we will not be interested in any non-tame inverse categories.
\end{notation}

\section{Subdiagrams} \label{sec:subdiagrams}

Let $\I$ be a tame category.
We are interested in full subcategories of $\I$, and we mean ``subcategory'' in the strict sense that the set of objects is a subset of the set of objects of $\I$.
We say that a full subcategory $J$ of $\I$ is downwards closed if, for any pair $x,y$ of objects in $\I$ with $y \prec x$, if $x$ is in $J$, then so is $y$.
The full downwards closed subcategories of $\I$ always form a poset $\mathsf{Sub}(\I)$, with an arrow $J \to J'$ if $J'$ is a subcategory of $J$. 

It is easy to see that the poset $\mathsf{Sub}(\I)$ has all limits and colimits. For example, given downwards closed full subcategories $J$ and $J'$, their product is given by taking the union of their sets of objects. We therefore write $J \cup J'$. 
Dually, coproducts are given by intersection and we can write $J \cap J'$.
An object $x$ of $\I$ generates a subcategory $\Set{y | y \preceq x}$, for which we write $\overline x$.

If $A : \I \to \ttfc$ is a Reedy fibrant diagram and $\ttfc$ has Reedy $\I$-limits, we can consider the functor
\begin{equation}
 \lim_{-}A : \mathsf{Sub}(\I) \to \ttfc
\end{equation}
which maps any downwards closed full subcategory $J \subseteq \I$ to $\lim_{J}A$, the Reedy limit of $A$ restricted to $J$.

\begin{lemma} \label{lem:continuous-functor}
 Let $\I$ be a tame category and $J,K$ two downwards closed subcategories of $\I$. 
 Then, the functor $\lim_{-}A$ maps the pullback square
 \begin{center}
\begin{tikzpicture}[align=left, node distance = 1.5cm]
  \node [color=black](A1) {$J \cup K$}; 
  \node [below of=A1, node distance=1cm](A2) {$J$}; 
  \node [right of=A1](A3) {$K$}; 
  \node [right of=A2](A4) {$J \cap K$}; 
  
  \draw[->] (A1) to node {} (A2);
  \draw[->] (A1) to node {} (A3);
  \draw[->] (A2) to node {} (A4);
  \draw[->] (A3) to node {} (A4);
\end{tikzpicture}
\end{center}
in $\mathsf{Sub}(\I)$ to a pullback square in $\ttfc$.
\end{lemma}
\begin{proof}
 For an object $X$, a cone $X \to A|_{J\cup K}$ corresponds to a pair of two cones, $X \to A|_J$ and $X \to A|_K$, which coincide on $J \cap K$.
\end{proof}

\begin{lemma} \label{lem:sub-gives-fibs}
 Under the same assumptions as before, the functor $\lim_ {-}A$ maps all morphisms to fibrations. In other word, if $K$ is a downwards closed subcategory of the inverse category $J$, then
 \begin{equation}
  \lim_J A \fib \lim_K A
 \end{equation}
 is a fibration.
\end{lemma}
\begin{proof}
 We only need to consider the case that $J$ has exactly one object that $K$ does not have, say $J \jdeq K +x$, because the composition of fibrations is a fibration (this is true even for ``infinite compositions'', with the same short proof as \cref{lem:acyclic_limit}).
 Further, we may assume that all objects of $J$ are predecessors of $x$, i.e.\ we have $\overline x \jdeq J$; otherwise, we could view $J \to K$ as a pullback of $\overline x \to \overline x - x$ and apply \cref{lem:continuous-functor}.
 
 The cone $\lim_K A \to A |_K$ gives rise to a cone $\lim_K A \to (A \circ U)|_{x \sslash K}$ (the morphism into $x \xrightarrow f y$ is given by the morphism into $y$), and we thereby get a morphism $m : \lim_K A \to M^A_x$.
 If we pull the fibration $A_x \fib M^A_x$ back along the morphism $m$, we get a fibration $P \fib \lim_K A$, and it is easy to see that $P \cong \lim_J A$.
\end{proof}

\begin{remark}
 The above proof yields a description in type-theoretic notation of the fibration $\lim_{K+x} A \fib \lim_K A$. It can be written as
 \begin{equation}
  \sm{k : \lim_K A} F^A(x, m(k)) \fib \lim_K A.  
 \end{equation}
 This remains true even if not all objects in $J$ are predecessors of $x$.
\end{remark}

\section{Equality Diagrams} \label{sec:equality}

Given any tame category $\I$ and a fixed type $B$ in $\ttfc$, the diagram 
$\I \to \ttfc$ that is constantly $B$ is, in general, not Reedy fibrant. 
Fortunately, the axioms of a type-theoretic fibration category allow us to define a \emph{fibrant replacement} (see, for example, Hoveys textbook~\cite{hovey2007model}).
We call the resulting diagram, which we construct explicitly, the \emph{equality diagram} of $B$ over $\I$.
We define by simultaneous induction:
\begin{enumerate}[label={(\roman*)},ref={\roman*}]
 \item a diagram $\Ee{B} : \I \to \ttfc$, the \emph{equality diagram}
 \item a cone $\eta : B \to \Ee{B}$ (i.e.\ a natural transformation from the functor that is constantly $B$ to $\Ee{B}$)
 \item a diagram $M^{\Ee{B}} : \I \to \ttfc$ (the diagram of matching objects)
 \item an auxiliary cone $\tilde \eta : B \to M^{\Ee{B}}$.
 \item a natural transformation $\iota : \Ee{B} \to M^{\Ee{B}}$
\end{enumerate}
such that $\iota \circ \eta$ equals $\tilde \eta$.

Assume that $i$ is an object in $\I$ such that the five components are defined for all predecessors of $i$. 
This is in particular the case if $i$ has no predecessors. We define the matching object
 $M^{\Ee{B}}_i \defeq \lim_{i \sslash \I} \Ee{B}$ as discussed in \cref{sec:ttfcs}.
The universal property of this limit yields
\begin{itemize}
 \item for every non-identity morphism $f : i \to j$, an arrow $\overline f : M^{\Ee{B}}_i \to \Ee{B}_j$, which lets us define $M^{\Ee{B}}(f)$ to be $\iota_j \circ \overline f$; and
 \item an arrow $\tilde \eta_i : B \to M^{\Ee{B}}_i$ such that, for every non-identity $f : i \to j$ as in the first point, we have that $\overline f \circ \tilde \eta_i$ equals $\eta_j$. 
\end{itemize}
We further define $\Ee{B}$ on objects by
\begin{equation} \label{eq:define-E-by-fibres}
 \Ee{B}_i \defeq \sm{m : M^{\Ee{B}}_i} \sm{x:B} \id{\tilde \eta_i(x)}{m}. 
\end{equation}
This allows us to choose the canonical projection map for $\iota_i$, and we can define $\Ee{B}$ on 
non-identity morphisms by
\begin{equation} 
 \Ee{B}(f) \defeq \overline f \circ \iota_i.
\end{equation}
Finally, we set
\begin{equation}
 \eta_i(x) \defeq (\tilde{\eta}_i(x), x , \refl {\tilde{\eta}_i(x)}).
\end{equation}
By construction, $\eta$, $\tilde \eta$, and $\iota$ satisfy the required naturality conditions.

\begin{lemma} \label{lem:eta-is-equiv}
 For all $i:\I$, the morphism $\eta_i : B \to \Ee{B}_i$ is a homotopy equivalence.
\end{lemma}
\begin{proof}
 This is due to the fact that
\begin{equation}
 \begin{alignedat}{2}
   \Ee{B}_i \quad & \jdeq \quad && \sm{m : M^{\Ee{B}}_i} \sm{x:B} \id{\tilde \eta_i(x)}{m} \\
   & \eqvsym && \sm{x:B} \sm{m : M^{\Ee{B}}_i} \id{\tilde \eta_i(x)}{m} \\
   & \eqvsym && B, 
 \end{alignedat}
\end{equation}
where the last step uses that the last two $\Sigma$-components have the form of a singleton.
\end{proof}

The proceeding lemma tells us that $\Ee{B}$ is levelwise homotopy equivalent to the constant diagram. The crux is that, unlike the constant diagram, $\Ee{B}$ is Reedy fibrant by construction, i.e.\ a \emph{fibrant replacement} in the usual terminology of model category theory.

\begin{lemma} \label{lem:Ef-is-equiv}
 For all morphisms $f$ in the category $\I$, the fibration $\Ee{B}(f)$ is a homotopy equivalence. 
\end{lemma}
\begin{proof}
 If $f:i \to j$ is a morphism in $\I$, we have $\Ee{B}(f) \circ \eta_i \jdeq \eta_j$ due to the naturality of $\eta$. The claim than follows by \cref{lem:eta-is-equiv} as homotopy equivalences satisfy ``2-out-of-3''.
\end{proof}

\section{The Equality Semi-simplicial Type} \label{sec:equality-sst}

Let $\deltplus$ be the category of non-zero finite ordinals and strictly increasing maps between them.
We write $\ordinal k$ for the objects, $\ordinal k \jdeq \Set{0,1,\ldots,k}$,
and $\incrmap k m$ for the hom-sets.
We can now turn to our main case of interest, which is the tame category $\I \jdeq \deltop$.  
In this case, we call $\Ee{B}$ the \emph{equality semi-simplicial type} of the (given) type $B$.
We could write down the first few values of $M^{\Ee{B}}_{\ordinal n}$ and $\Ee{B}_{\ordinal n}$ explicitly. 
However, these type expressions would look rather bloated. 
More revealing might be the homotopically equivalent presentation in \cref{fig:e-m-ideal}.
\begin{figure}[H]
 \begin{alignat*}{3}
   &M^{\Ee{B}}_{\ordinal 0} \quad && \jdeq \quad && \unit \\
   &{\Ee{B}}_{\ordinal 0} \quad && \eqvsym \quad && B \\
   &M^{\Ee{B}}_{\ordinal 1} \quad && \eqvsym \quad && B \times B \\
   &\Ee{B}_{\ordinal 1} \quad && \eqvsym \quad && \sm{b^1,b^2 : B} \id {b^1}{b^2}\\
   &M^{\Ee{B}}_{\ordinal 2} \quad && \eqvsym \quad && \sm{b^1, b^2, b^3 : B} (\id{b^1}{b^2}) \times (\id{b^2}{b^3}) \times (\id{b^1}{b^3}) \\
   &\Ee{B}_{\ordinal 2} \quad && \eqvsym \quad && \sm{b^1, b^2, b^3 : B} \sm{p: \id{b^1}{b^2}} \sm{q : \id{b^2}{b^3}} \sm{ r : \id{b^1}{b^3}}     p \ct q = r. 
 \end{alignat*}
 \caption{The ``nicer'' formulation of the equality semi-simplicial type.
 The equivalences can be shown easily using the contractibility of singletons.} \label{fig:e-m-ideal}
\end{figure}
We think of $\Ee{B}_{\ordinal 0}$ as the type of points, $\Ee{B}_{\ordinal 1}$ as the type of lines (between two points), and of $\Ee{B}_{\ordinal 2}$ as the type of triangles (with its faces). The ``boundary'' of a triangle, as represented by $M_{\ordinal 2}$, consists of three points with three lines, and so on.
In general, we think of $\Ee{B}_{\ordinal n}$ as (the type of) $n$-dimensional tetrahedra, while $M^{\Ee{B}}_{\ordinal n}$ are their ``complete boundaries''.
In principle, we could have defined $\Ee{B}$ in a way such that \cref{fig:e-m-ideal} are judgmental equalities rather than only equivalences: the stated types could be completed to form a Reedy fibrant diagram. However, we do not think that this is possible using a definition that is as uniform and short as the one above.
Already for $\Ee{B}_{\ordinal 3}$, it seems unclear what the best formulation would be if we wanted to follow the presentation of \cref{fig:e-m-ideal}.
In general, such a construction would most likely make use of Street's orientals~\cite{Street1987283}.

For any $\ordinal n$, the co-slice category $\ordinal n \slash \deltop$ is a poset. This is a consequence of the fact that all morphisms in $\deltplus$ are monic.
We have the forgetful functor ${U : \ordinal n \slash \deltop \to \deltop}$.
Further, $\ordinal n \slash \deltop$ is isomorphic to the poset $\mathcal P_+(\ordinal n)$ of nonempty subsets of the set $\ordinal n \jdeq \{0,1,\ldots,n\}$, where we have an arrow between two subsets if the first is a superset of the second.
The downwards closed full subcategories of $\ordinal n \slash \deltop$ correspond to downwards closed subsets of $\mathcal P_+(\ordinal n)$.
If $S$ is such a downwards closed subset, we write $\lim_{S}(\Ee{B} \circ U)$, omitting the implied functor $S \to \ordinal n \slash \deltop$.

Any set $s \subseteq \ordinal n$ generates such a downwards closed set for which we write $\overline s \defeq \mathcal P_+(s)$. 
For $k \in s$, we write $\overline s_{-k}$ for the set that we get if we remove exactly two sets from $\overline s$, namely $s$ itself and the set $s - k$ (i.e.\ $s$ without the element $k$).
We call $\lim_{\overline {\ordinal n}_{-k}}(\Ee{B} \circ U)$ the \emph{$k$-th $n$-horn}.

\begin{mainlemma} \label{mainlem:horn-filler-trivial}
 For any $n\geq 1$ and $k \in \ordinal n$, call the fibration from the full $n$-dimensional tetrahedron to the $k$-th $n$-horn
 \begin{equation}
  \lim_{\overline {\ordinal n}}(\Ee{B} \circ U) \; \fib \; \lim_{\overline {\ordinal n}_{-k}}(\Ee{B} \circ U)
 \end{equation}
 a \emph{horn-filler fibration}.
 All horn-filler fibrations are homotopy equivalences.
\end{mainlemma}
\begin{remcor}[Types are Kan complexes]
 As both Steve Awodey and an anonymous reviewer of have pointed out to me, \cref{mainlem:horn-filler-trivial} can be seen as a simplicial variant of Lumsdaine's~\cite{lumsdaine_weakOmegaCatsFromITT} and van den Berg-Garner's~\cite{bg:type-wkom} result that types are weak $\omega$-groupoids. 
 Both of these (independent) articles use Batanin's~\cite{batanin1998monoidal} definition, slightly modified by Leinster~\cite{leinster2002survey}, of a weak $\omega$-groupoid.
 
 Let us make the construction of a simplicial weak $\omega$-groupoid, i.e.\ of a Kan complex, concrete. 
 We can do this for the assumed type-theoretic fibration category $\ttfc$ as long as it is locally small (i.e.\ all hom-sets are sets).
 As before, we can without loss of generality
 assume that the type we want to consider lives in the empty context, i.e.\ is given by an object $B$. 
 We can define a semi-simplicial set
 \begin{align}
  & S : \deltop \to \mathsf{Set} \\
  & S_{\ordinal n} \defeq \ttfc(\unit, \Ee{B}_{\ordinal n}). 
 \end{align}
 For a morphism $f$ of $\deltop$, the functor $S$ is given by simply composing with $\Ee{B}(f)$.
 
 Shulman's \emph{acyclic fibration lemma}~\cite[Lemma 3.11]{shulman_inversediagrams}, applied on the result of our 
 \cref{mainlem:horn-filler-trivial}, 
 gives us sections of all horn-filler fibrations.
 Therefore, $S$ satisfies the Kan condition.
 By a result Rourke and Sanderson~\cite{rourke1971delta} (see also McClure~\cite{mcclure2013semisimplicial} for a combinatorical proof), such a semi-simplicial set can be given the structure of a Kan \emph{simplicial} set, an incarnation of a weak $\omega$-groupoid.

 To get the result that types in HoTT are Kan complexes, we simply take $\ttfc$ to be the syntactic category of HoTT, where we have to assume strict $\eta$ for $\Pi$, $\Sigma$ and $\unit$.
 This allows us to say very concretely that the terms of the types that we can write down form a Kan complex.
\end{remcor}
\begin{proof}[Proof of \cref{mainlem:horn-filler-trivial}]
 Fix $\ordinal n$.
 We show more generally that, for any $s \subseteq \ordinal n$ with cardinality $|s| \geq 2$ and $k \in s$, the fibration
  \begin{equation}
  \lim_{\overline s}(\Ee{B} \circ U) \fib \lim_{\overline s_{-k}}(\Ee{B} \circ U)
 \end{equation}
 is an equivalence. 
 Note that $\lim_{\overline s}(\Ee{B} \circ U)$ is isomorphic to $\Ee{B}_{\ordinal {|s|-1}}$.
 
 The proof is performed by induction on the cardinality of $s$. 
 If $s$ has only one element apart from $k$, 
 then $\overline s_{-k}$ is the one-object category $\{\{k\}\}$
 and we have 
 \begin{equation}
  \lim_{\{\{k\}\}}(\Ee{B} \circ U) \cong \Ee{B}_{\ordinal 0}.
 \end{equation}
 The statement then follows from \cref{lem:Ef-is-equiv}.
 
 Let us explain the induction step.
 The inclusions $\{\{k\}\} \subseteq \overline s_{-k} \subset \overline s$ give rise to a triangle 
\begin{center}
\begin{tikzpicture}[align=left, node distance=3.5cm]

  \node [color=black](A1) {$\lim_{\overline s}(\Ee{B} \circ U)$}; 
  \node [right of=A1](A2) {$\lim_{\overline s_{-k}}(\Ee{B} \circ U)$}; 
  \node [below of=A2, node distance=1.5cm](A3) {$\lim_{\{\{k\}\}}(\Ee{B} \circ U)$}; 
  
  \draw[->>] (A1) to node {} (A2);
  \draw[->>] (A1) to node {} (A3);  
  \draw[->>] (A2) to node {} (A3);
\end{tikzpicture}
\end{center}
 of fibrations. The top horizontal fibration is the one of which we want to prove that it is an equivalence. 
 Using ``2-out-of-3'' and the fact that the left (diagonal) fibration is an equivalence by \cref{lem:Ef-is-equiv}, it is sufficient to show that the right vertical fibration is an equivalence.
 To do this, we decompose it into $2^{|s|-1}-1$ fibrations, each of which can be viewed as the pullback of a smaller horn-filler fibration:
 
 Consider the set $\mathcal P_+(s-k)$ of those nonempty subsets of $s$ that do not contain $k$. The number of those is $2^{|s|-1} - 1$.
 We label those sets as $\alpha_1, \alpha_2, \ldots, \alpha_{2^{|s|-1} - 1}$, where the order is arbitrary with the only condition that their cardinality is nondecreasing, i.e.\ $i < j$ implies $|\alpha_i| < |\alpha_j|$.
 
 We further define $2^{|s|-1}$ subsets of $\mathcal P_+(s)$, named $S_0, S_1, \ldots, S_{2^{|s|-1}}$.
 Define $S_0$ to be $\{\{k\}\}$.
 Then, define $S_{i}$ to be $S_{i-1}$ with two additional elements, namely $\alpha_i$ and $\alpha_i \cup \{k\}$.
 In this process, every element of $\mathcal P_+(s)$ is clearly added exactly once. In particular, $S_{2^{|s|-1}} \jdeq \overline s$ and $S_{2^{|s|-1}-1} \jdeq \overline s_{-k}$.
 Further, all $S_i$ are downwards closed, which is easily seen to be the case by induction on $i$: it is the case for $i \jdeq 0$, and in general, $S_i$ contains all proper subsets of $\alpha_i \cup \{k\}$ due to the single ordering condition that we have put on the sequence $(\alpha_j)$.
 
 It is easy to see that
 \begin{align}
  & S_i \; \jdeq \; S_{i-1} \cup \overline{\alpha_i \cup \{k\}} \\
  & \overline{\alpha_i \cup \{k\}}_{-k}  \; \jdeq \;   S_{i-1} \cap \overline{\alpha_i \cup \{k\}}.
 \end{align}
 By \cref{lem:continuous-functor}, we thus have a pullback square

 \begin{center}
\begin{tikzpicture}[align=left, node distance=1.5cm]
  \node [color=black](A1) {$\lim_{S_i}(\Ee{B} \circ U)$}; 
  \node [right of=A1, node distance =3.5cm](A2) {$\lim_{\overline{\alpha_i \cup \{k\}}}(\Ee{B} \circ U)$}; 
  \node [below of=A1](A3) {$\lim_{S_{i-1}}(\Ee{B} \circ U)$}; 
  \node [below of=A2](A4) {$\lim_{\overline{\alpha_i \cup \{k\}}_{-k}}(\Ee{B} \circ U)$}; 
  
  \draw[->>] (A1) to node {} (A2);
  \draw[->>] (A1) to node {} (A3);
  \draw[->>] (A2) to node {} (A4);
  \draw[->>] (A3) to node {} (A4);
\end{tikzpicture}
 \end{center}

 \noindent
For $i \leq 2^{|s|-1} - 2$, the right vertical morphism is a homotopy equivalence by the induction hypothesis.
As acyclic fibrations are stable under pullback, the left vertical morphism is one as well.
As the composition of equivalences is an equivalence, we conclude that 
\begin{equation}
 \lim_{\overline s_{-k}}(\Ee{B} \circ U)
 \fib
 \lim_{\{\{k\}\}}(\Ee{B} \circ U)
\end{equation}
is indeed an equivalence.
\end{proof}

\begin{remark}
 Recall that a simplicial object $X : \oppo\Delta \to \mathcal D$ satisfies the \emph{Segal condition} (see~\cite{Segal:Segal}) if the ``fibration''
 \begin{equation} \label{eq:inf:segal}
  X_{\ordinal n} \to \underbrace{X_{\ordinal 1}\times_{X_{\ordinal 0}}X_{\ordinal 1}\times_{X_{\ordinal 0}} \ldots \times_{X_{\ordinal 0}}X_{\ordinal 1}}_{n \textit{ factors}}
 \end{equation}
 is an equivalence.
 In our situation, it looks as if it was easy to check the Segal condition; more precisely, a shorter argument than the one in the proof could show that \emph{all} the fibrations of the form~\eqref{eq:inf:segal}) are homotopy equivalences.
 Our construction with the sequence $\alpha_1, \alpha_2, \ldots, \alpha_{2^{|s|-1} - 1}$ seems to contain a ``manual'' proof of the fact that checking this form of the Segal condition would be sufficient.
\end{remark}

\section{Fibrant Diagrams of Natural Transformations} \label{sec:nat-trans}

Let us first formalise what we mean by the ``type of natural transformations between two diagrams''.
If $I$ is a tame category and $D,E : I \to \ttfc$ are Reedy fibrant diagrams, the exponential $E^D : I \to \ttfc$ in the functor category $\ttfc^I$ exists and is Reedy fibrant~\cite[Theorem 11.11]{shulman_inversediagrams} and thus has a limit in $\ttfc$.
What we are interested in is the more general case that $D$ might not be fibrant, but we also do not need any exponential.\footnote{The author expects that the exponential $E^D$ exists and is fibrant even if only $E$ is fibrant (note that $D$ is automatically at least \emph{pointwise} fibrant, as all objects in $\ttfc$ are fibrant by definition). This would lead to an alternative representation of the same construction, but the author has  decided to use the less abstract one presented here as it seems to give a more direct argument.} 
On a more abstract level, what we want to do can be described as follows. For any downwards closed subcategory of $I$, we consider the exponential of $D$ and $E$ restricted to this subcategory, and take its limit. 
We basically construct approximations to the ``type of natural transformations'' from $D$ to $E$ which, in fact, corresponds to the limit of these approximations, should it exist.
Fortunately, it is easy to do everything ``by hand'' on a very basic level.

We write $\truncCat I$ for the underlying partially ordered set of $I$ that we get if we make any two parallel arrows equal (we ``truncated'' all hom-sets).
This makes sense even if $I$ is not inverse, but if it is, then so is $\truncCat I$.
There is a canonical functor ${\bproj -}_I : I \to \truncCat I$.
As the objects of $I$ are the same as those of $\truncCat I$, we omit this functor when applied to an object, i.e.\ for $i \in I$ we write $i \in \truncCat I$ instead of $\bproj i_I \in \truncCat I$.

\begin{definition}[Diagram of Natural Transformations]
Given an inverse category $I$, a diagram $D : I \to \ttfc$ and a fibrant diagram $E : I \to \ttfc$ with 
\begin{equation}
 E_i \jdeq \sm{m : M^E_i} F^E_{(i,m)} 
\end{equation}
as introduced in \cref{notation:match-and-fibre}, we define a fibrant diagram $N : \truncCat I \to \ttfc$ together with a natural transformation
\begin{equation}
 v : \big((N \circ {\bproj -}_I) \times D\big) \to E
\end{equation}
simultaneously, where $(N \circ {\bproj -}_I) \times D$ is the functor $I \to \ttfc$ that is given by taking the product pointwise.

Assume $i$ is an object in $I$. Assume further that we have defined both $N$ and $v$ for all predecessors of $i$ (i.e.\ $N$ is defined on $\Set{x \in \truncCat I | x \prec i}$ and $v$ is defined on $\Set{x \in I | x \prec i}$).
$v$ then gives rise to a map 
\begin{equation} \label{eq:overline-v}
 \overline v : \lim_{\Set{x \in \truncCat I | x \prec i}}N \times D_i \to M^E_i.
\end{equation}

Using \cref{lem:change-notation-poset}, we have $\lim_{\Set{x \in \truncCat I | x \prec i}}N \; \cong \; \lim_{i \sslash \truncCat I}(N \circ U) \; \cong \; M^N_i$.
We define $N_i \jdeq \sm{m : M^N_i} F^N_{(i,m)}$ by choosing the fibre over $m$ to be 
\begin{equation} \label{eq:nat-trans-total-def}
 F^N_{(i,m)} \defeq \prd{d : D_i} F^E(i, \overline{v}(m,d)).
\end{equation}
This definition also gives a canonical morphism $v_i : N_i \times D_i \to E_i$ which extends $v$.
\end{definition}

Let us apply this construction to define the type of constant functions between types $A$ and $B$ in the way that we already suggested in \cref{fig:const-as-nat-trans} on page \pageref{fig:const-as-nat-trans}.
First, we define the \emph{$\ordinal 0$-coskeleton} of the diagram that is constantly $A$, which we also have referred to as the \emph{trivial diagram} over $A$, 
as the functor  
$\Aa{A}  : \deltop \to \ttfc$ as follows. 
For objects, it is simply given by 
\begin{equation}
 \Aa{A}_{\ordinal k} \defeq \underbrace{A \times A \times \ldots \times A}_{(k+1) \textit{ factors}}. 
\end{equation}
If we view an element of $\Aa{A}_{\ordinal i}$ as a function $\ordinal i \to A$, for a map $f : \incrmap i j$ we get $\Aa{A} (f) : \Aa{A}_{\ordinal j} \to \Aa{A}_{\ordinal i}$ by composition with $f$.
We then define the functor $\Nn_{A,B} : \truncCat\deltop \to \ttfc$ via the above construction as the ``fibrant diagram of natural transformations'' from $\Aa{A} $ to $\Ee{B}$. 
Note that $\truncCat \deltop$ is isomorphic to $\oppo{\omega}$. 
Using the homotopy equivalent formulation of $\Ee{B}$ stated in \eqref{fig:e-m-ideal} and the definitions of $\const$ and $\cohcond$ of \cref{sec:motivation}, we get 
\begin{equation}
 \Nn_{A,B}(\ordinal 0) \eqvsym (A \to B)
\end{equation}
as well as 
\begin{equation}
 \Nn_{A,B}(\ordinal 1) \eqvsym  \sm{f : A \to B} \const_f
\end{equation}
and 
\begin{equation}
 \Nn_{A,B}(\ordinal 2) \eqvsym  \sm{f : A \to B} \sm{c : \const_f} \cohcond_{f,c}.
\end{equation}
We want to stress the intuition that we think of functions with an infinite tower of coherence condition by introducing the following notation:

\begin{definition}[$A \toomega B$] \label{not:toomega}
 Given types $A$ and $B$, we write $A \toomega B$ synonymously for $\lim_{\truncCat\deltop}\Nn_{A,B}$.
\end{definition}
We usually omit the indices of $\Nn_{A,B}$ and just write $\Nn$, provided that $A,B$ are clear from the context.
This allows us write $\Nn_{\ordinal n}$ instead of $\Nn_{A,B}(\ordinal n)$.

Analogously to \cref{not:toomega}, let us write the following:
\begin{definition}[$A \ton n B$] \label{not:ton}
 Given types $A$ and $B$ and a (usually finite) number $n$, we write $A \ton n B$ synonymously for $\Nn_{\ordinal n}$. 
 To enable a uniform presentation, we define $A \ton {-1} B$ to be the unit type.
\end{definition}

We are now able to make the main goal, as outlined in \cref{subsec:outline-of-inf}, precise:
we will construct a function $(\proptrunc A \to B) \to (A \toomega B)$ and prove that it is a homotopy equivalence.
For now, let us record that we can get a function $B \to (A \toomega B)$. 
In the following definition, we use the cones $\eta : B \to \Ee{B}$ and $\tilde\eta : B \to M^{\Ee{B}}$ from \cref{sec:equality}. 
\begin{definition}[Canonical function $s : B \to (A \toomega B)$]\label{def:canonical-function}
 Define a cone $\gamma : B \to \Nn$ which maps $b:B$ to the function that is ``judgmentally constantly $b$'', in the following way.
 First, notice that the matching object $M^{\Nn}_{\ordinal n}$ is simply $\Nn_{\ordinal {n-1}}$ (due to the fact that $\truncCat \deltop$ is a total order).
 Assume we have already defined the component $\gamma_{\ordinal {n-1}} : B \to \Nn_{\ordinal {n-1}}$ such that $\overline v(\gamma_{\ordinal {n-1}}(b),x) \jdeq \tilde\eta_{\ordinal n}(b)$, with $\overline v$ as in \eqref{eq:overline-v}, for all $x : \Aa{A}_{\ordinal n}$. We can then define $\gamma_{\ordinal n}(b)$ by giving an element of $F^{\Nn}(\ordinal n,\gamma_{\ordinal {n-1}}(b))$, but that expression evaluates to $\prd{x:\Aa{A}_{\ordinal n}}\sm{x:B}\id{\tilde\eta_{\ordinal n}(x)}{\tilde\eta_{\ordinal n}(b)}$. Thus, we can take $\gamma_{\ordinal n}(b)$ to be
 \begin{equation}
  \gamma_{\ordinal n}(b) \defeq \big(\gamma_{\ordinal {n-1}}(b) , \lam z (b , \refl {\tilde\eta_{\ordinal {n-1}}(b)})\big).
 \end{equation}
 It is straightforward to check that the condition $\overline v(\gamma_{\ordinal n},x) \jdeq \tilde\eta_{\ordinal {n+1}}(b)$ is preserved.
 Define the function 
  $s : B \to (A \toomega B)$
 to be $\lim_{\truncCat\deltop}\gamma$, the arrow that is induced by the universal property of the limit.
\end{definition}

\section{Extending Semi-Simplicial Types} \label{sec:extending-sst}

In this section, we first define the category $\ccat$. 
We can then view $\opccat$ as an extension of $\deltop$, as $\deltop$ can be embedded into $\opccat$, and this embedding has a retraction $R$ with the property that the co-slice $c \slash \opccat$ is always isomorphic to $R_c \slash \deltop$.
With the help of this category, we can describe precisely how we want to apply our ``expanding and contracting'' strategy.
The definition of $\ccat$ is motivated by the proofs of \cref{satz:special-case-0,satz:special-case-1}, and this will become clear when we show how exactly we use $\ccat$, see especially \cref{fig:N}.

\begin{definition}[$\ccat$]
 Let $\ccat$ be the following category. 
 For every object $\ordinal k$ of $\deltplus$ (i.e.\ every natural number $k$),
 and every number $i \in \ordinal {k+1}$, we have an object $\cob k i$.
 Given objects $\cob k i$ and $\cob m j$, we define $\ccat \left( \cob k i , \cob m j \right)$ to a subset of the set of maps $\deltplus(\ordinal k , \ordinal m)$. 
 It is given by 
 \begin{equation}
  \ccat \left( \cob k i , \cob m j \right) \defeq \Set{f :  \incrmap k m | \alpha(k,m,i,j) }
 \end{equation}
 where the condition $\alpha$ is defined as
 \begin{equation} \label{eq:inf:alpha}
   \alpha(k,m,i,j) \defeq
    \begin{cases}
      f(x) \jdeq x \mbox{ for all } x < i, \mbox{ and } f(x) > x \mbox{ for all } x \geq i & \mbox{if } i < j \\
      f(x) \jdeq x \mbox{ for all } x < i & \mbox{if } i \jdeq j \\
      \bot & \mbox{if } i > j.
    \end{cases}
 \end{equation}
\end{definition}
What will be useful for us is the opposite category $\opccat$.
A part of it, namely the subcategory $\Set{\cob k i \in \opccat | k \leq 3}$, can be pictured as shown in \cref{fig:opccat}. 
We only draw the ``generating'' arrows $\cob {m+1} j \to \cob m i$.
\begin{figure}[ht]
\begin{center}
\begin{tikzpicture}[node distance=\textwidth/8]
  \node (xA0) {$\cob 0 0$}; 
  \node [right of=xA0] (xA1) {$\cob 0 1$};
  
  \node [above of=xA0] (xB0) {$\cob 1 0$}; 
  \node [right of=xB0] (xB1) {$\cob 1 1$};
  \node [right of=xB1] (xB2) {$\cob 1 2$};
  \node [right of=xB2] (xB3) {}; 
  
  \node [above of=xB0] (xC0) {$\cob 2 0$}; 
  \node [right of=xC0] (xC1) {$\cob 2 1$};
  \node [right of=xC1] (xC2) {$\cob 2 2$};
  \node [right of=xC2] (xC3) {$\cob 2 3$};
  
  \node [above of=xC0] (xD0) {$\cob 3 0$}; 
  \node [right of=xD0] (xD1) {$\cob 3 1$};
  \node [right of=xD1] (xD2) {$\cob 3 2$};
  \node [right of=xD2] (xD3) {$\cob 3 3$};
  \node [right of=xD3] (xD4) {$\cob 3 4$};
  
  \node [above of=xD0, node distance = 0.7cm] (xE0) {\ldots}; 
  \node [above of=xD1, node distance = 0.7cm] (xE1) {\ldots}; 
  \node [above of=xD2, node distance = 0.7cm] (xE2) {\ldots}; 
  \node [above of=xD3, node distance = 0.7cm] (xE3) {\ldots}; 
  \node [above of=xD4, node distance = 0.7cm] (xE4) {\ldots}; 
  
  \draw[->, transform canvas={xshift=-0.25ex}] (xB0) to node {} (xA0);
  \draw[->, transform canvas={xshift=0.25ex}] (xB0) to node {} (xA0);
  \draw[->] (xB1) to node {} (xA0);
  \draw[->] (xB1) to node {} (xA1);
  \draw[->] (xB2) to node {} (xA0); 
  \draw[->] (xB2) to node {} (xA1);

  \draw[->, transform canvas={xshift=-0.5ex}] (xC0) to node {} (xB0);
  \draw[->, transform canvas={xshift=0ex}] (xC0) to node {} (xB0);
  \draw[->, transform canvas={xshift=0.5ex}] (xC0) to node {} (xB0);
  \draw[->] (xC1) to node {} (xB0);
  \draw[->, transform canvas={xshift=-0.25ex}] (xC1) to node {} (xB1);
  \draw[->, transform canvas={xshift=0.25ex}] (xC1) to node {} (xB1);
  \draw[->] (xC2) to node {} (xB0); 
  \draw[->] (xC2) to node {} (xB1);
  \draw[->] (xC2) to node {} (xB2);
  \draw[->] (xC3) to node {} (xB0); 
  \draw[->] (xC3) to node {} (xB1); 
  \draw[->] (xC3) to node {} (xB2);

  \draw[->, transform canvas={xshift=-0.75ex}] (xD0) to node {} (xC0);
  \draw[->, transform canvas={xshift=-0.25ex}] (xD0) to node {} (xC0);
  \draw[->, transform canvas={xshift=0.25ex}] (xD0) to node {} (xC0);
  \draw[->, transform canvas={xshift=0.75ex}] (xD0) to node {} (xC0);
  \draw[->] (xD1) to node {} (xC0);
  \draw[->, transform canvas={xshift=-0.5ex}] (xD1) to node {} (xC1);
  \draw[->, transform canvas={xshift=0ex}] (xD1) to node {} (xC1);
  \draw[->, transform canvas={xshift=0.5ex}] (xD1) to node {} (xC1);
  \draw[->] (xD2) to node {} (xC0); 
  \draw[->] (xD2) to node {} (xC1);
  \draw[->, transform canvas={xshift=-0.25ex}] (xD2) to node {} (xC2);
  \draw[->, transform canvas={xshift=0.25ex}] (xD2) to node {} (xC2);
  \draw[->] (xD3) to node {} (xC0); 
  \draw[->] (xD3) to node {} (xC1); 
  \draw[->] (xD3) to node {} (xC2);
  \draw[->] (xD3) to node {} (xC3);
  \draw[->] (xD4) to node {} (xC0); 
  \draw[->] (xD4) to node {} (xC1); 
  \draw[->] (xD4) to node {} (xC2);
  \draw[->] (xD4) to node {} (xC3);
\end{tikzpicture}
\caption{The category $\opccat$}\label{fig:opccat}
\end{center} 
\end{figure}
The idea is that the full subcategory of objects $\cob m 0$ is exactly $\deltplus$, 
and that every object $\cob m i$ in $\ccat$ receives exactly one arrow for every $\incrmap k m$. 
We make this precise as follows:
\begin{lemma}\label{lem:coslices-equivalent}
 The canonical embedding $\deltop \hookrightarrow \opccat$, defined by $\ordinal m \mapsto \cob m 0$, has a retraction
 \begin{align}
  &R : \opccat \to \deltop \\
  &R(\cob m j) \defeq \ordinal m
 \end{align}
 and, for all objects $\cob m j$ in $\opccat$, the functor that $R$ induces on the co-slice categories
 \begin{equation}
  \cob m j \slash \opccat \; \to \; \ordinal m \slash \deltop
 \end{equation}
 is an isomorphism of categories. 
\end{lemma}
\begin{proof}
 It is clear that $\deltop \hookrightarrow \opccat \xrightarrow R \deltop$ is the identity on $\deltop$.
 For any $\cob m j$, fix an object $\ordinal k$ in $\deltop$ and take a morphism $f : \incrmap k m$.
 There is exactly one $i$ such that the condition $\alpha(k,m,i,j)$ in \eqref{eq:inf:alpha} is fulfilled.
 This proves the second claim.
\end{proof}

Let us extend the functor $\Aa{A} : \deltop \to \ttfc$ (see \cref{sec:nat-trans}) to the whole category $\opccat$. 
Assume that a type $A$ is given.
We want to define a diagram $\AA{A}$ that extends $\Aa{A}$.
This corresponds to the point where, in \cref{sec:motivation}, we had assumed that a point $\apoint : A$ was given, in other words, we had added $(\apoint : A)$ to the context.
We do the same here.
Recall that we write $\fibslice \ttfc A$ for the type-theoretic fibration category with fixed context $A$, as explained in \cref{sec:ttfcs}.
The diagram that we define is a functor
\begin{equation}
 \AA{A} : \opccat \to \fibslice \ttfc A.
\end{equation}
In order to be closer to the type-theoretic notation and to hopefully increase readability,
we write objects of $\fibslice \ttfc A$ simply as 
$B(\apoint)$ if they are of the form $\sm{a:A}B(a) \fib A$.
This uses that we can do the whole construction fibrewise, i.e.\ that we can indeed assume a fixed but arbitrary $\apoint : A$ ``in the context''.
Of course, objects in $\fibslice \ttfc A$ of the form $A \times B \fib A$ are simply denoted by $B$.

Using this notation, we define $\AA{A}$ on objects by
\begin{equation}
 \AA{A}(\cob m j) \defeq \underbrace{A \times A \times \ldots \times A}_{(m+1-j) \text{ factors}},
\end{equation}
for which we simply write $A^{m+1-j}$.
Given $\cob m j \xrightarrow{f} \cob k i$ in $\opccat$, we thus need to define a map
$\AA{A} (f) : A^{m+1-j} \to A^{k+1-i}$.
As in the definition of $\Aa{A} $, the map $f : \incrmap {k}{m}$ gives rise to a function
$\overline f : A^{m+1} \to A^{k+1}$ by ``composition''.
We define $\AA{A} (f)$ as the composite

\begin{center}
\begin{tikzpicture}[node distance=1.4cm]
  \node (i2j2) {$A^{m+1-j}$}; 
  \node [below of=i2j2] (j2i2j2) {$A^{j} \times A^{m+1-j}$};
  \node [below of=j2i2j2] (j1i1j1) {$A^{i} \times A^{k+1-i}$};
  \node [below of=j1i1j1] (i1j1) {$A^{k + 1 - i}$};
  \draw[biggerArrow] (i2j2) to node [right] {$\vec a \mapsto (\underbrace{\apoint, \apoint, \ldots, \apoint}_{j \text{ times } \apoint}, \vec a)$} (j2i2j2);
  \draw[->>] (j2i2j2) to node [right] {$\overline f$} (j1i1j1);
  \draw[->>] (j1i1j1) to node [right] {$\snd$} (i1j1);
\end{tikzpicture}
\end{center}

We have a diagram $\Ee{B} \circ R : \opccat \to \ttfc$, which we can (pointwise) pull back along $A \fib \unit$, which gives us a diagram that we call $\EE{B} : \opccat \to \fibslice \ttfc A$.
This diagram is Reedy fibrant. 
With the construction of \cref{sec:nat-trans}, we can define $\NN : \truncCat \opccat \to \fibslice \ttfc A$ to be the ``fibrant diagram of natural transformations'' from $\AA{A} $ to $\EE{B}$.

We can picture $\NN$ on the subcategory $\Set{\cob m j \in \truncCat \opccat | m \leq 2}$ as shown in \cref{fig:N}.
For readability, we use the homotopy equivalent representation of the values of $\Ee{B}$ as shown in \cref{fig:e-m-ideal}.
Further, we only write down the values of $F^{\EE{B}}$ (i.e.\ the fibres) instead of the full expression $\EE{B}(\cob m j) \jdeq \sm{t : M^{\EE{B}}(\cob m j)} F^{\EE{B}}(\cob m j, t)$. 
This means that, for example, $\const_f \fib (f : A \to B)$ stands for the projection $\sm { f : A \to B} \const_f \fib (A \to B)$. 
The reader is invited to make a comparison with \cref{satz:special-case-1}. Recall that, in the proof of \cref{satz:special-case-1}, we have started with the $\Sigma$-component $f_1$. In the ``expanding'' part, we have added the pair of $f$ and $c_1$, which (together) form a contractible type, as well as the pair of $c$ and $d_1$, and $c_2$ and $d_3$. We have also used that the types of $d$ and $d_2$ are, in the presence of the other $\Sigma$-components, contractible. Then, in the ``retracting'' part, we have used that the types of $d_3$ and $d_1$ are contractible, and that $c_1$ and $d_2$, as well as $f_1$ and $c_2$, form pairs of two other contractible types.

\begin{figure}[ht]
\begin{center}
\begin{tikzpicture}[node distance=\textwidth/4]
  \node (A0) {$f : A \to B$}; 
  \node [right of=A1] (A1) {$f_1 : B$};
  
  \node [above of=A0, node distance=2cm] (B0) {$c : \const_f$}; 
  \node [right of=B0] (B1) {$c_1 : \prd{a:A} \id{f(a)}{f_1}$};
  \node [right of=B1] (B2) {$c_2 : \id{f(\apoint)}{f_1}$};
  \node [right of=B2] (B3) {}; 
  
  \node [above of=B0, node distance=3cm] (C0) {$d : \cohcond_{f,c}$}; 
  \node [above of=B1, node distance=3.5cm] (C1) {$d_1 : \prd{a^1 a^2:A} \id{c(a^1,a^2) \cdot c_1(a^2)}{c_1(a^1)}$};
  \node [above of=B2, node distance=3cm] (C2) {$d_2 : \prd{a:A} c(\apoint,a) \cdot c_1(a) = c_2$};
  \node [above of=B3, node distance=2.5cm] (C3) {$d_3 : \id{c(\apoint,\apoint) \cdot c_1(\apoint)}{c_2}$};
  
  \draw[->>] (B0) to node {} (A0);
  \draw[->>] (B1) to node {} (A0);
  \draw[->>] (B1) to node {} (A1);
  \draw[->>] (B2) to node {} (A0); 
  \draw[->>] (B2) to node {} (A1);

  \draw[->>] (C0) to node {} (B0);
  \draw[->>] (C1) to node {} (B0);
  \draw[->>] (C1) to node {} (B1);
  \draw[->>] (C2) to node {} (B0); 
  \draw[->>] (C2) to node {} (B1);
  \draw[->>] (C2) to node {} (B2);
  \draw[->>] (C3) to node {} (B0); 
  \draw[->>] (C3) to node {} (B1); 
  \draw[->>] (C3) to node {} (B2);
\end{tikzpicture}
\caption{The diagram $\NN$ in readable (homotopy equivalent) representation; only the three lowest levels (the images of $\cob m j$ with $m \leq 2$) are drawn. 
Note that we use that same identifiers as in the proofs of \cref{satz:special-case-0,satz:special-case-1}.} \label{fig:N}
\end{center}
\end{figure}

To compare $\NN$ with $\Nn$, first note that $\Nn : \deltop \to \ttfc$ can be pulled back along $A \fib \unit$ pointwise and yields a diagram $\deltop \to \fibslice \ttfc A$.
This diagram is identical (pointwise isomorphic) to the diagram that we get if we first pull back the diagrams $\Aa{A}$ and $\Ee{B}$, and then take the ``fibrant diagram of natural transformations''.
Further, as ``limits commute with limits'', the limit of this diagram is, in $\fibslice \ttfc A$, isomorphic to the pullback of $A \toomega B$ along $A \fib \unit$.
It is thus irrelevant at which point in the construction we ``add $(\apoint : A)$ to the context'', i.e.\ at which point we switch from $\ttfc$ to the slice over $A$.
This allows us to compare constructions in $\fibslice \ttfc A$ and $\ttfc$, by implicitly pulling back the latter.
As it is easy to see, $\NN$ extends $\Nn$ in this sense (i.e.\ $\NN(\cob m 0)$ is the pullback of $\Nn_{\ordinal m}$ along $A \fib \unit$).

Recall that we have defined a cone $\gamma : B \to \Nn$ and an arrow $s : B \to (A \toomega B)$ in \cref{def:canonical-function}.
Exploiting that $\gamma_{\ordinal n}(b)$ was defined in a way that makes it completely independent of the ``argument'' $x : \Aa{A}_{\ordinal n}$, and using \cref{lem:coslices-equivalent}, we can extend $\gamma$ to a cone
 $\overline{\gamma}: B \to \NN$,
essentially by putting $\overline{\gamma}_{\cob m j} \defeq \gamma_{\ordinal m}$.
This gives a morphism 
\begin{equation} \label{eq:s-extended}
 \overline s : B \to \lim_{\truncCat\opccat}\NN
\end{equation}
which extends $s$, in the sense that (the pullback of) $s$ is the composition
\begin{equation} \label{eq:s-overline-s} 
 B \xrightarrow{\overline s} \lim_{\truncCat\opccat}\NN \xrightarrow{\mathsf{pr}} \lim_{\truncCat\deltop}\Nn,
\end{equation} 
with $\mathsf{pr}$ coming from the embedding $\truncCat\deltop \hookrightarrow \truncCat\opccat$ and the fact that the restriction of $\NN$ to $\{\cob m 0\}$ is $\Nn$ (pulled back along $A \fib \unit$; note that the codomain of $\mathsf{pr}$ is implicitly pulled back as well).
Further, observing that $\NN(\cob 0 1)$ is canonically equivalent to $B$ (as used in \cref{fig:N}), the composition
\begin{equation}\label{eq:overline-s-is-section}
 B \xrightarrow{\overline s} \lim_{\truncCat\opccat}\NN \xrightarrow{\mathsf{pr'}} \NN(\cob 0 1) \xrightarrow \sim B
\end{equation}
is the identity on $B$.

\section{The Main Theorem} \label{sec:maintheorem}

The preparations of the previous sections allow us prove our main result.
We proceed analogously to our arguments for the special cases in \cref{sec:motivation}:
\cref{lem:N-horn-contr} and \cref{cor:main} show that certain fibrations are homotopy equivalences, i.e.\ that certain types are contractible. 
This is then used in \cref{lem:main} to perform the ``expanding and contracting'' argument, which shows that, if we assume a point in $A$, the function $s$ from \cref{def:canonical-function} is a homotopy equivalence.
Admittedly, especially \cref{lem:N-horn-contr} requires extensive calculations.

We need to keep the extremely simple statement of \cref{lem:change-notation-poset} in mind: the limit of $\NN \circ U$ restricted to $z \sslash \truncCat {\opccat}$ is isomorphic to the limit of $\NN$ restricted to $\Set{y \in \truncCat{\opccat} | y \prec z}$.
We prefer the slightly more concise first notation.

For the following statement, note that $\NN(\cob m j)$ is the same as $\lim_{\cob m j \slash \truncCat \ccat} \NN \circ U$.
\begin{lemma}\label{lem:N-horn-contr}
 The fibration
 \begin{equation}\label{eq:fib-with-set-notation}
  \NN(\cob m j) \fib \lim_{\Set{x \in \truncCat \opccat | x \prec \cob m j, x \not\jdeq \cob{m-1}{j-1} }} \NN 
 \end{equation}
 is a homotopy equivalence for any $\ordinal m$ and $j$.
\end{lemma}
\begin{proof}
There is a single morphism in $\opccat\left(\cob m j, \cob{m-1}{j-1}\right)$.
For the category ${\cob m j \sslash \opccat}$ where this morphism is removed,
we write ${\cob m j \sslash \opccat - \cob{m-1}{j-1}}$.
The fibration \eqref{eq:fib-with-set-notation} can then be written as
 \begin{equation}
  \NN(\cob m j) \fib \lim_{\cob m j \sslash \truncCat \ccat - \cob {m-1}{j-1}} \NN \circ U.
 \end{equation}
By construction of $\NN$, we have a natural transformation $v : (\NN \circ \bproj-_{\opccat}) \times \AA{A} \to \EE{B}$, which gives rise to a morphism
 \begin{equation} \label{eq:w}
  w : \big( \lim_{\cob m j \sslash \truncCat\opccat - \cob{m-1}{j-1}} \NN \circ U \big)  \times \AA{A} (\cob m j) \; \to \; \lim_{\cob m j \sslash \opccat - \cob{m-1}{j-1}} \EE{B} \circ U.
 \end{equation}
 Consider the diagram shown in \cref{fig:bigdiagram}, in which $Q$ is defined to be the pullback. 
\begin{figure}[ht]
\begin{center}
 \begin{tikzpicture}[align=left, node distance=1.5cm]
  \node [color=black](C3) {$Q$}; 
  \node [below of=C3](C2) {}; 
  \node [below of=C2](C1) {$\big( \lim_{\cob m j \sslash \truncCat\opccat - \cob{m-1}{j-1}} \NN \circ U \big)  \times \AA{A} (\cob m j)$}; 

  \node [right of=C3, node distance=3.0cm](CC3){};

  \node [below of=CC3, node distance=4.1cm](B3) {$\sm{k : M^{\EE{B}}(\cob m j)} F^{\EE{B}}(\cob m j,k)$}; 
  \node [below of=B3](B2) {$M^{\EE{B}}(\cob m j)$}; 
  \node [below of=B2](B1) {$\lim_{\cob m j \sslash \opccat - \cob {m-1} {j-1}} \EE{B} \circ U$}; 

  \node [right of=B2, node distance=4.0cm](A2) {$\sm{t : M^{\EE{B}}(\cob {m-1} {j-1})} F^{\EE{B}}(\cob {m-1} {j-1},t)$}; 
  \node [below of=A2](A1) {$M^{\EE{B}}(\cob {m-1} {j-1})$};

  \draw[biggerArrow, dashed] (C3) to node {} (C1);
  \draw[->>] (B3) to node {} (B2);
  \draw[->>] (B2) to node {} (B1);
  \draw[->>] (A2) to node {} (A1);
  \draw[biggerArrow, dashed, bend left=40] (C3) to node {} (B3);
  \draw[biggerArrow, bend right=20] (C1) to node [below left] {$w$} (B1); 
  \draw[->>] (B2) to node {} (A2);
  \draw[->>] (B1) to node [above] {$\proj$} (A1);

  \draw [dashed] ($(C3.center) + (0.2cm,-0.8cm)$) --++ (0.6cm,-0.2cm) --++ (0,0.6cm);
\end{tikzpicture}
\caption{Derivation of a homotopy equivalence}\label{fig:bigdiagram}
\end{center} 
\end{figure}
 The right part (everything without the leftmost column) of that diagram comes from applying the functor
 $\lim_{-} (\EE{B} \circ U)$
 to the diagram in $\mathsf{Sub}(\cob m j \slash \opccat)$ that is shown in \cref{fig:smalldiagram}.
\begin{figure}[ht]
\begin{center}
  \begin{tikzpicture}[align=left, node distance=1.5cm]
  \node [color=black](C3) {}; 
  \node [below of=C3](C2) {}; 
  \node [below of=C2](C1) {}; 

  \node [right of=C3](B3) {$\cob m j \slash \opccat$}; 
  \node [right of=C2](B2) {$\cob m j \sslash \opccat$}; 
  \node [right of=C1](B1) {$\cob m j \sslash \opccat - \cob {m-1} {j-1}$}; 

  \node [right of=B2, node distance=4cm](A2) {$\cob {m-1} {j-1} \slash \opccat$}; 
  \node [right of=B1, node distance=4cm](A1) {$\cob {m-1} {j-1} \sslash \opccat$};

  \draw[->] (B3) to node {} (B2);
  \draw[->] (B2) to node {} (B1);
  \draw[->] (A2) to node {} (A1);
  \draw[->] (B2) to node {} (A2);
  \draw[->] (B1) to node {} (A1);
 \end{tikzpicture}
\caption{A small diagram in $\mathsf{Sub}(\cob m j \slash \opccat)$.
 This uses the principle that, in an inverse category $\I$ with a morphism $k : x \to y$, the categories $k \sslash (x \sslash \I)$ and $y \sslash \I$ are isomorphic.}\label{fig:smalldiagram}
\end{center} 
\end{figure}
 
 In \cref{fig:bigdiagram}, the fibration labelled $\proj$ comes of course from 
 \begin{equation}
  \left(\cob m j \sslash \opccat - \cob {m-1} {j-1}\right) \supset \left(\cob {m-1} {j-1} \sslash \opccat\right),
 \end{equation}
 as shown in \cref{fig:smalldiagram}.
 We give it a name solely to make referencing it easier.
 Our goal is to derive a representation of $Q$.
 As the right square is a pullback square by \cref{lem:continuous-functor}, we have
 \begin{equation}
  M^{\EE{B}}(\cob m j) \; \cong \; \smonlybig{t : \lim_{\cob m j \sslash \opccat - \cob {m-1} {j-1}} \EE{B}} F^{\EE{B}}(\cob {m-1} {j-1}, \proj (t)).
 \end{equation}
 Using this, can write the top expression of the middle column as
 \begin{equation}
 \begin{aligned}
  & \smonlybig{k : M^{\EE{B}}(\cob m j)} F^{\EE{B}}(\cob m j,k) \\
  \eqvsym \quad & \smonlybig{t : \lim_{\cob m j \sslash \opccat - \cob {m-1} {j-1}} \EE{B}}   \smonlybig{n : F^{\EE{B}}(\cob {m-1} {j-1}, \proj (t))} F^{\EE{B}}(\cob m j,(t,n)).
 \end{aligned}
 \end{equation}
 The pullback $Q$ is thus
 \begin{equation}
  \begin{alignedat}{2}
  &\smonlybig{p : \lim_{
  \cob m j \sslash \truncCat\opccat - \cob {m-1} {j-1}}
  \NN \circ U } &&  \smonlybig{a : \AA{A} (\cob m j)}  \\
  &&&\smonlybig{n : F^{\EE{B}}(\cob {m-1} {j-1}, \proj(w(p,a))) } \\
  &&&\phantom{\Sigma (} F^{\EE{B}}(\cob m j,(w(p,a),n)).
  \end{alignedat}
 \end{equation}
 The composition of the two vertical fibrations in the middle column is a homotopy equivalence by \cref{mainlem:horn-filler-trivial} and \cref{lem:coslices-equivalent}.
 As acyclic fibrations are stable under pullback, the fibration 
 \begin{equation}
  Q \; \fib \; \big(\lim_{\cob m j \sslash \truncCat\opccat - \cob {m-1} {j-1}} \NN\big) \times \AA{A} (\cob m j)
 \end{equation}
 is a homotopy equivalence as well.
 Function extensionality implies that a family of contractible types is contractible (i.e.\ that acyclic fibrations are preserved by $\Pi$), and we get that the first projection
 \begin{equation} \label{eq:havetrivfib}
 \begin{tikzpicture}[align=left, node distance=5cm]
  \node [color=black, align=left](C3) {$\smonlybig{p : \lim_{\Set{x \in \truncCat\opccat | x \prec \cob m j , x \not\jdeq \cob {m-1} {j-1}}} \NN }$ \\ $\phantom{\Sigma} \prd{a : \AA{A} (\cob m j)}
      \sm{n : F^{\EE{B}}(\cob {m-1} {j-1}, \proj(w(p,a)))} 
                      F^{\EE{B}}(\cob m j,(w(p,a),n))$}; 
  \node [below of=C3, node distance=2.5cm](C1) {$\lim_{\Set{x \in \truncCat\opccat | x \prec \cob m j , x \not\jdeq \cob {m-1} {j-1}}} \NN$}; 

  \draw[->>] (C3) to node {} (C1);
 \end{tikzpicture}
 \end{equation} 
 is also a homotopy equivalence. The lemma is therefore shown if we can prove that the domain of the above fibration \eqref{eq:havetrivfib}, a rather lengthy expression, is homotopy equivalent to $\NN(\cob m j)$.
 Our first step is to apply the distributivity law~\eqref{eq:distributivity} to transform this expression to
 \begin{equation}\label{eq:N-N-split}
 \begin{aligned}
   &\smonlybig{p : \lim_{\Set{x \in \truncCat\opccat | x \prec \cob m j , x \not\jdeq \cob {m-1} {j-1}}} \NN } \\
   &
   \smonlybig{n : \prd{a : \AA{A} (\cob m j)}
     F^{\EE{B}}(\cob {m-1} {j-1}, \proj(w(p,a))) } \\
   &\phantom{\Sigma\big(} \prd{a:\AA{A}(\cob m j)} F^{\EE{B}}(\cob m j,(w(p,a),n(a))).
 \end{aligned}
 \end{equation}
 When we look at the following square, in which $w$ is the map \eqref{eq:w}, $w'$ is induced by the natural transformation $v$ in the same way as $w$, and $\proj$, $\proj'$ come from the restriction to subcategories,
 \begin{equation} \label{eq:commuting-square}
 \begin{tikzpicture}[align=left, node distance=5cm]
  \node [color=black](X1) {$\big(\lim_{
  \cob m j \sslash \truncCat\opccat - \cob {m-1} {j-1}
  } \NN \circ U\big) \times \AA{A} (\cob m j)$}; 
  \node [below of=X1, node distance=2cm](X2) {$\big(\lim_{
  \cob {m-1} {j-1} \sslash \truncCat\opccat 
  } \NN \circ U\big) \times \AA{A} (\cob {m-1} {j-1})$}; 
  \node [right of=X1, node distance=6cm] (Y1) {$\lim_{\cob m j\sslash \opccat - \cob {m-1} {j-1}} \EE{B} \circ U$};
  \node [right of=X2, node distance=6cm] (Y2) {$\lim_{\cob {m-1} {j-1}\sslash \opccat} \EE{B} \circ U$};

  \draw[->>] (X1) to node [left] {$\proj'$} (X2);
  \draw[->>] (Y1) to node [right] {$\proj$} (Y2);
  \draw[biggerArrow] (X1) to node [above] {$w$} (Y1);
  \draw[biggerArrow] (X2) to node [above] {$w'$} (Y2);
 \end{tikzpicture}
 \end{equation} 
 we can see that it commutes due to the naturality of the natural transformation $v$. In particular, note that $\AA{A}$ maps the single morphism $\cob m j \to \cob {m-1} {j-1}$ to the identity on $A^{m+1-j}$. 
 This is exactly what is needed to see that the second line of~\eqref{eq:N-N-split} corresponds to the ``missing $\Sigma$-component'' $\NN(\cob {m-1} {j-1})$ in the limit of the first line. 
 Hence, the first and the second line can be ``merged'' and are equivalent to $\lim_{\cob m j \sslash \truncCat\opccat} \NN \circ U$, in other words, $M^{\NN}(\cob m j)$. 
 Comparing the third line of \eqref{eq:N-N-split} with the definition of the ``fibrant diagram of natural transformations'' (see \eqref{eq:nat-trans-total-def}), we see that \eqref{eq:N-N-split} is indeed equivalent to $\NN(\cob m j)$, as required.
\end{proof}

By pullback (\cref{lem:continuous-functor} and preservation of homotopy equivalences along pullbacks), we immediately get:
\begin{corollary} \label{cor:main}
 Let $D$ be a downwards closed subcategory of $\opccat$ which does not contain the objects $\cob m j$ and $\cob {m-1} {j-1}$, but all other predecessors of $\cob m j$.
 The full subcategory of $\opccat$ which has all the objects of $D$ and the objects $\cob {m-1} {j-1}$, $\cob m j$ (for which we write $D +\cob {m-1} {j-1} + \cob m j$) is also downwards closed and the fibration
 \begin{equation}
  \lim_{D + \cob {m-1} {j-1} + \cob m j} \NN \; \fib \; \lim_{D} \NN
 \end{equation}
 is a homotopy equivalence. \qed
\end{corollary}

\cref{cor:main} is the crucial statement that summarises all of our efforts so far.
We can use it to ``add and remove'' contractible $\Sigma$-components in the same way as we did it in the motivating examples (\cref{sec:motivation}). 
More precisely, we exploit that we can group together components of $\opccat$ in two different ways.
Our main lemma is the following:
\begin{mainlemma} \label{lem:main}
 Given types $A, B$, recall that we have defined $s : B \to (A \toomega B)$ in \cref{def:canonical-function}.
 Assume further that we are given a point $\apoint : A$ (i.e.\ we regard $s$ as a morphism in $\fibslice \ttfc A$ instead of $\ttfc$).
 Then, the function $s$ is a homotopy equivalence.
\end{mainlemma}

\begin{proof}
Using the point $\apoint$, we define $\NN$ and $\overline s: B \to \lim_{\truncCat \opccat}\NN$ as before
in \eqref{eq:s-extended}, and consider the following:
\begin{equation} \label{eq:diagram-3-4-2}
\begin{tikzpicture}[align=left, node distance=2.5cm]
  \node (firstB) {$B$}; 
  \node [right of = firstB] (firstN) {$\lim_{\truncCat\opccat} \NN$}; 
  \node [right of = firstN](secondN) {$\NN(\cob 0 1)$}; 
  \node [right of = secondN] (secondB) {$B$}; 
  \node [below of = firstN, node distance=1.5cm] (thirdN) {$A \toomega B$}; 
  
  \draw[biggerArrow] (firstB) to node [above] {$\overline s$} (firstN);
  \draw[->>] (firstN) to node [above] {$\mathsf{pr'}$} (secondN);
  \draw[biggerArrow] (secondN) to node [above] {$\sim$} (secondB);
  \draw[biggerArrow] (firstB) to node [below left] {$s$} (thirdN);
  \draw[biggerArrow] (firstN) to node [right] {$\mathsf{pr}$} (thirdN);
\end{tikzpicture}
\end{equation}
The commutativity of the triangle on the left is given by \eqref{eq:s-overline-s}. Our first goal is to show that the fibration $\mathsf{pr'}$ is a homotopy equivalence.

\pgfmathsetmacro{\nodevert}{14} 
\pgfmathsetmacro{\nodehor}{9} 
\pgfmathsetmacro{\dist}{3.5} 
\pgfmathsetmacro{\dashedheigth}{3} 

\pgfmathsetmacro{\angle}{atan(\nodevert/\nodehor)}
\pgfmathsetmacro{\sang}{\angle+90}
\pgfmathsetmacro{\eang}{\angle+270}
\pgfmathsetmacro{\fstlength}{\dashedheigth / sin(\angle)}
\pgfmathsetmacro{\sndlength}{\fstlength + cot(\angle) * 2 * \dist}

\newcommand{\drawincompgrouping}[1]{
  \draw
    ($(#1) + (\sang:\dist mm)$) arc (\sang:\eang:\dist mm);
  \draw [dashed]
    ($(#1) + (\sang:\dist mm)$) --++ (\sang-90:\fstlength mm)
    ($(#1) + (\eang:\dist mm)$) --++ (\sang-90:\sndlength mm); 
}

\begin{figure}[ht]
\subfloat[The sequence $D_0 \subset D_1 \subset D_2 \subset \ldots$ starts with $\{\cob 0 1\}$]{ 
\label{subfig:D-s}
\begin{tikzpicture}[node distance=\nodehor mm]
  \node (xA0) {$\cob 0 0$}; 
  \node [right of=xA0] (xA1) {$\cob 0 1$};
  
  \node [above of=xA0, node distance = \nodevert mm] (xB0) {$\cob 1 0$}; 
  \node [right of=xB0] (xB1) {$\cob 1 1$};
  \node [right of=xB1] (xB2) {$\cob 1 2$};
  \node [right of=xB2] (xB3) {}; 
  
  \node [above of=xB0, node distance = \nodevert mm] (xC0) {$\cob 2 0$}; 
  \node [right of=xC0] (xC1) {$\cob 2 1$};
  \node [right of=xC1] (xC2) {$\cob 2 2$};
  \node [right of=xC2] (xC3) {$\cob 2 3$};
  
  \node [above of=xC0, node distance = \nodevert mm] (xD0) {$\cob 3 0$}; 
  \node [right of=xD0] (xD1) {$\cob 3 1$};
  \node [right of=xD1] (xD2) {$\cob 3 2$};
  \node [right of=xD2] (xD3) {$\cob 3 3$};
  \node [right of=xD3] (xD4) {$\cob 3 4$};
  
  \node [above of=xD0, node distance = \nodevert mm] (xE0) {$\cob 4 0$}; 
  \node [right of=xE0] (xE1) {$\cob 4 1$};
  \node [right of=xE1] (xE2) {$\cob 4 2$};
  \node [right of=xE2] (xE3) {$\cob 4 3$};
  \node [right of=xE3] (xE4) {$\cob 4 4$};
  \node [right of=xE4] (xE5) {$\cob 4 5$};
  
  \draw[black] \convexpath{xA0,xB1}{\dist mm};
  \draw[black] \convexpath{xB0,xC1}{\dist mm};
  \draw[black] \convexpath{xC0,xD1}{\dist mm};
  \draw[black] \convexpath{xD0,xE1}{\dist mm};
  \draw[black] \convexpath{xB2,xC3}{\dist mm};
  \draw[black] \convexpath{xC2,xD3}{\dist mm};
  \draw[black] \convexpath{xD2,xE3}{\dist mm};
  \draw[black] \convexpath{xD4,xE5}{\dist mm};

  \drawincompgrouping{xE0}
  \drawincompgrouping{xE2}
  \drawincompgrouping{xE4}
\end{tikzpicture}
}
\, \vrule \, 
\subfloat[The sequence $D_0' \subset D_1' \subset D_2' \subset \ldots$ starts with $\{\cob n 0\}$]{ %
\label{subfig:D-prime-s}
\begin{tikzpicture}[node distance=\nodehor mm]
  \node (xA0) {$\cob 0 0$}; 
  \node [right of=xA0] (xA1) {$\cob 0 1$};
  
  \node [above of=xA0, node distance = \nodevert mm] (xB0) {$\cob 1 0$}; 
  \node [right of=xB0] (xB1) {$\cob 1 1$};
  \node [right of=xB1] (xB2) {$\cob 1 2$};
  \node [right of=xB2] (xB3) {}; 
  
  \node [above of=xB0, node distance = \nodevert mm] (xC0) {$\cob 2 0$}; 
  \node [right of=xC0] (xC1) {$\cob 2 1$};
  \node [right of=xC1] (xC2) {$\cob 2 2$};
  \node [right of=xC2] (xC3) {$\cob 2 3$};
  
  \node [above of=xC0, node distance = \nodevert mm] (xD0) {$\cob 3 0$}; 
  \node [right of=xD0] (xD1) {$\cob 3 1$};
  \node [right of=xD1] (xD2) {$\cob 3 2$};
  \node [right of=xD2] (xD3) {$\cob 3 3$};
  \node [right of=xD3] (xD4) {$\cob 3 4$};
  
  \node [above of=xD0, node distance = \nodevert mm] (xE0) {$\cob 4 0$}; 
  \node [right of=xE0] (xE1) {$\cob 4 1$};
  \node [right of=xE1] (xE2) {$\cob 4 2$};
  \node [right of=xE2] (xE3) {$\cob 4 3$};
  \node [right of=xE3] (xE4) {$\cob 4 4$};
  \node [right of=xE4] (xE5) {$\cob 4 5$};
  
  \draw[black] \convexpath{xA1,xB2}{\dist mm};
  \draw[black] \convexpath{xB1,xC2}{\dist mm};
  \draw[black] \convexpath{xC1,xD2}{\dist mm};
  \draw[black] \convexpath{xD1,xE2}{\dist mm};
  \draw[black] \convexpath{xC3,xD4}{\dist mm};
  \draw[black] \convexpath{xD3,xE4}{\dist mm};

  \drawincompgrouping{xE1}
  \drawincompgrouping{xE3}
  \drawincompgrouping{xE5}
  
\end{tikzpicture}
} 
\caption{
Two infinite sequences of downwards closed full subcategories of $\truncCat \opccat$, constructed in the proof of \cref{lem:main}: the first starts with $D_0 \defeq \{\cob 0 1\}$. 
In each step, exactly two objects are added, paired as shown in the left drawing.
The second sequence starts with $D_0' \defeq \{\cob n 0\}$ (the leftmost column), i.e.\ $D_0'$ is isomorphic to $\oppo\omega$.
The pairings are shown in the right drawing.
The reader who has read through the proof of \cref{lem:main} is invited to combine the current figure with \cref{fig:N} in order to reconstruct the proof of \cref{satz:special-case-1}. 
}
\end{figure}

Consider the set $S \defeq \Set{(m,j) \in \N^2 | \text{$j$ is even and $j \leq m+1$}}$.
A pair $(m,j)$ is in $S$ if and only if $\cob m j$ is an object in an ``odd column'' of $\opccat$ in \cref{fig:opccat} on page \pageref{fig:opccat} (where we consider the leftmost column the ``first'').
Define a total order on $S$ by letting $(k,i)$ be smaller than $(m,j)$ if either $k+i < m+j$ or ($k+i \jdeq m+j$ and $i < j$).
We represent this total order by an isomorphism $f : \N_+ \to S$ (where $\N_+$ are the positive natural numbers) which has the property that $f(n)$ is always smaller than $f(n+1)$.
Write $f_1(n)$ and $f_2(n)$ for the first respectively the second projection of $f(n)$.

Let us define a sequence 
 $D_0 \subset D_1 \subset D_2 \subset D_3 \subset \ldots$
of full subcategories of $\truncCat \opccat$ by
\begin{align}
 D_0 &\defeq \{\cob 0 1\} \\
 D_n &\defeq D_{n-1} \, + \, \cob {f_1(n)} {f_2(n)} \, + \, \cob {f_1(n)+1} {f_2(n)+1}.
\end{align}
This construction illustrated in \cref{subfig:D-s}.
It is easy to see that every object $\cob m j$ is added exactly once, i.e.\ it is either $\cob 0 1$ or it is of the form $\cob {f_1(n)} {f_2(n)}$ or of the form $\cob {f_1(n)+1} {f_2(n)+1}$ for exactly one $n$. We have chosen the total order on $S$ in such a way that every $D_n$ is a downwards closed full subcategory of $\truncCat\opccat$.
Applying \cref{cor:main}, we get a sequence
\begin{equation}
 \lim_{D_0}\NN \trivfibre \lim_{D_1}\NN \trivfibre \lim_{D_2}\NN \trivfibre \lim_{D_3}\NN \trivfibre \ldots
\end{equation}
of acyclic fibrations.
\cref{lem:acyclic_limit} then shows that the canonical map
\begin{equation} 
 \lim_{\opccat}\NN \trivfib \lim_{D_0}\NN 
\end{equation}
is an acyclic fibration.
As $\lim_{D_0}\NN$ is simply $\NN(\cob 0 1)$, this proves that $\mathsf{pr'}$ is indeed a homotopy equivalence.

Next, we want to show the same about $\mathsf{pr}$. We proceed very similarly.
This time, we define $S' \defeq \Set{(m,j) \in \N^2 | \text{$j$ is odd and $j \leq m+1$}}$.
A pair $(m,j)$ is consequently in $S'$ if and only if $\cob m j$ is an object in an ``even'' column of \cref{fig:N}.
As before, we define an isomorphism $f' : \N_+ \to S'$, and 
define a sequence 
 $D_0' \subset D_1' \subset D_2' \subset D_3' \subset \ldots$
of full subcategories of $\truncCat \opccat$ by
\begin{align}
 D_0' &\defeq \{\cob m 0\} \qquad \text{(i.e.\ the full subcategory corresponding to $\truncCat\deltop$)}\\
 D_n' &\defeq D_{n-1}' \, + \, \cob {f_1(n)} {f_2(n)} \, + \, \cob {f_1(n)+1} {f_2(n)+1}.
\end{align}
We illustrate the construction of this sequence in \cref{subfig:D-prime-s}, which the reader is encourage to compare with \cref{subfig:D-s}.
Again, every object $\cob m j$ is added exactly once, and every $D_n$ is downwards closed.
\cref{cor:main} and \cref{lem:acyclic_limit} then tell us that $\lim_{\truncCat\opccat}\NN \fib \lim_{\{\cob m 0\}}\NN$ is an acyclic fibration. Hence, $\mathsf{pr}$ is indeed a homotopy equivalence, as claimed.

We take another look at the diagram \eqref{eq:diagram-3-4-2}. The composition of the three horizontal arrows is the identity by \eqref{eq:overline-s-is-section}. But homotopy equivalences satisfy ``2-out-of-3'', and we can conclude that $\overline s$ is an equivalence. Using ``2-out-of-3'' again, we see that $s$ is an equivalence as well. 
\end{proof}

It is straightforward to define what it mean for a type-theoretic fibration category to have propositional truncations by imitating the characterisation given in \cref{sec:motivation}.
We now show:
\begin{lemma} \label{lem:inf-lastBeforeMain}
 If $\ttfc$ has propositional truncations,
 then the canonical function 
 \begin{equation}
  s : B \to (A \toomega B), 
 \end{equation}
 viewed as a morphism in $\ttfc$, is a homotopy equivalence assuming $\proptrunc A$. 
 More precisely, we can construct a function
 \begin{equation}
  \proptrunc A \to \isequiv(s)
 \end{equation}
 in $\ttfc$.
\end{lemma}
\begin{proof}
 We have shown in \cref{lem:main} that $s$ is a homotopy equivalence in $\fibslice \ttfc A$, i.e.\ if we pull back its domain and codomain along $A \fib \unit$.
 In $\ttfc$, this means that
 \begin{equation}
  \lam {(a,b)} (a , s(b)) : A \times B \; \to \; A \times (A \toomega B)
 \end{equation}
 is an equivalence, but this implies 
 \begin{equation}
  A \to \isequiv(s).
 \end{equation}
 The claim then follows from the ordinary universal property of the propositional truncation.
\end{proof}
This allows us to prove our main result:
\begin{theorem}[General universal property of the propositional truncation] \label{thm:result}
 Let $\ttfc$ be a type-theoretic fibration category that satisfies function extensionality, 
 has propositional truncations, and Reedy $\oppo\omega$-limits. 
 Let $A$ and $B$ be two types, i.e.\ objects in $\ttfc$. 
 Using the canonical function $s : B \to (A \toomega B)$ as defined in \cref{def:canonical-function},
 we can construct a function
 \begin{equation}
  (\proptrunc A \to B) \; \to \; (A \toomega B) ,
 \end{equation}
 and this function is a homotopy equivalence.
\end{theorem}
\begin{proof}
 From \cref{lem:inf-lastBeforeMain} we can conclude, just as in the special cases in \cref{sec:motivation}, that 
 \begin{align}
  (\proptrunc A \to B)  \; &\to \; \left(\proptrunc A \to (A \toomega B)\right) \\
  f \; &\mapsto \; \lam x s(f(x)) \label{eq:inf-simple-homo-equiv}
 \end{align}
 is a homotopy equivalence.

 This is not yet what we aim for. We need a statement corresponding to the infinite case of \cref{lem:ignore-trunc-A}, i.e.\ we need to prove that $\proptrunc A \to (A \toomega B)$ is equivalent to $A \toomega B$.
 To do this, we consider the diagram $\Pp : \truncCat{\deltop} \to \ttfc$, defined on objects by 
 \begin{equation}
  \Pp_{\ordinal k} \defeq \proptrunc A \to \Nn_{\ordinal k},
 \end{equation}
and on morphisms by 
\begin{equation} \label{eq:inf-result-isFib}
 \Pp(g) \defeq \lam {(h : \proptrunc A \to \Nn_{\ordinal k})} \lam x \Nn(g)(h(x)). 
\end{equation}
 
 Paolo Capriotti has pointed out that $\Pp$ is Reedy fibrant, and this is a crucial observation.
 As $\Pp$ is defined over a poset, it is enough to show that \eqref{eq:inf-result-isFib} is a fibration for every $g$.
 Our argument is the following:
 The maps in both directions which are used to prove the distributivity law~\eqref{eq:distributivity} are \emph{strict} inverses, i.e.\ their compositions (in both orders) are judgmentally equal to the identities. 
 This means that every $\Pp_i$ is isomorphic to a $\Sigma$-type, where we ``distribute'' $\proptrunc A$ over the 
 $\Sigma$-components.
 From this representation, it is clear that $\Pp(g)$ is always a fibration, as fibrations are closed under composition with isomorphisms.
 
 Because of \cref{lem:ignore-trunc-A} (and the fact that the equivalence there is defined uniformly), there is a natural transformation $\kappa : \Pp \to \Nn$ which is levelwise a homotopy equivalence. By the definition of $\ttfc$ having Reedy $\oppo\omega$-limits, the resulting arrow between the two limits, that is
 \begin{equation} \label{eq:inf-aux}
  \lim_{\truncCat{\deltop}}(\kappa) : \;  \left(\proptrunc A \to (A \toomega B)\right) \; \to \; (A \toomega B),
 \end{equation}
 is a homotopy equivalence as well.
 To conclude, we simply compose \eqref{eq:inf-simple-homo-equiv} and \eqref{eq:inf-aux}.
\end{proof}

\section{Finite Cases} \label{sec:infconstancy-finite}

If $B$ is an $n$-type for some finite fixed number $n$, the higher coherence conditions should intuitively become trivial.
This is obvious for the representation of $\Nn$ and $\Ee{B}$ given in \cref{fig:e-m-ideal,fig:N}, although admittedly not for our actual definition of $\Ee{B}$ in \cref{sec:equality} (and the corresponding definition of $\Nn$ and $\NN$) where it requires a little more thought. This is our main goal for this section.
For the presentation, we assume that the type theory in question has a universe $\UU$, although this assumption is not strictly necessary. 
After this, it will be easy to see that the universal properties of the propositional truncation with an $n$-type as codomain, for any externally fixed number $n$, can be formulated and proved in standard syntactical homotopy type theory.

We start by reversing the statement that ``singletons are contractible'':
\begin{lemma}\label{mainlem:singletons-everywhere} 
 Let us assume that $A$ is a type, $B : A \to \UU$ is a family, and $C : (\sm {a:A} B(a)) \to \UU$ a second family.
 The following are logically equivalent:
 \begin{enumerate}[label={(\roman*)},ref={\roman*}]
  \item For any $a:A$, there is a point $b_a : B(a)$ and a homotopy equivalence \label{item:inf:coc1} 
  \begin{equation}
    \eqv{C(a,b)}{(\id b {b_a})}.   
  \end{equation}
  \item The canonical projection \label{item:inf:coc2}
  \begin{equation}
   \fst : \big(\sm{a:A}\sm{b:B(a)}C(a,b)\big) \to A 
  \end{equation}
  is a homotopy equivalence.
 \end{enumerate}
\end{lemma}
\begin{proof}
 The direction \implref{item:inf:coc1}{item:inf:coc2} is an obvious consequence from the contractibility of singletons. 
For the other direction, recall that, for any type $X$ and families $Y, Z : X \to \UU$, a map
\begin{equation} \label{eq:fibrewise-map}
 f : \prd{x:X} \left(Y(x) \to Z(x)\right)
\end{equation}
 is a \emph{fibrewise (homotopy) equivalence} if each $f(x) : Y(x) \to Z(x)$ is a homotopy equivalence~\cite[Chapter 4.7]{HoTTbook}.
Given~\eqref{eq:fibrewise-map}, there is a canonical way to define a map on the total spaces
\begin{equation}
 \mathsf{total}(f) : \sm{x:X}Y(x) \, \to \, \sm{x:X}Z(x).
\end{equation}
Then, $\mathsf{total}(f)$ is a homotopy equivalence if an only if $f$ is a fibrewise homotopy equivalence~\cite[Thm. 4.7.7]{HoTTbook}. 
Using this result, we derive a very short proof of~\implref{item:inf:coc2}{item:inf:coc1}:

We fix $a:A$ and assume~\eqref{item:inf:coc2} which implies that $\sm{b:B(a)}C(a,b)$ is contractible. 
This gives us the required $b_a$ and allows us to define a map
\begin{equation}
 g : \prd{b:B(a)} \left(C(a,b) \to \id{b_a}{b}\right).
\end{equation}
Clearly, $\mathsf{total}(g)$ is a homotopy equivalence as it is a map between contractible types.
Hence, $g$ is a fibrewise homotopy equivalence.
\end{proof}

We are now ready to show that, in the case of $n$-types, the higher ``fillers for complete boundaries'' become homotopically simpler and simpler, and finally trivial.
\begin{lemma} \label{lem:movenumbers}
 Let $n \geq -2$ be a number and $B$ be a type in $\ttfc$. 
 Consider the equality semi-simplicial type $\Ee{B} : \deltop \to \ttfc$ of $B$.
 For an object $\ordinal k$ of $\deltop$, we can consider the fibration $\Ee{B}_{\ordinal k} \fib M^{\Ee{B}}_{\ordinal k}$.
 We know that, by definition, the fibre over $m : M^{\Ee{B}}_{\ordinal k}$ is simply $\sm{x:B} \id{\tilde \eta_{\ordinal k}(x)}{m}$.
 
 If $B$ is an $n$-type, then, for any object $\ordinal k$ of $\deltop$, all these fibres are $(n-k)$-truncated (or contractible, if this difference is below $-2$). 
\end{lemma}
\begin{remark}
 The other direction of \cref{lem:movenumbers} should also hold, as $M^{\Ee{B}}_{\ordinal k}$ should be equivalent to $\sm{b:B}\Omega^{k}(B,b)$. We do neither prove nor require this direction here.
\end{remark}
\begin{proof}[Proof of \cref{lem:movenumbers}]
 The statement clearly holds for $\ordinal k \jdeq \ordinal 0$, as the matching object $M^{\Ee{B}}_{\ordinal k}$ will in this case be the unit type.
 We assume that the statement holds for $\ordinal k$ and show it for $\ordinal {k+1}$.
 Recall our notation from \cref{sec:equality-sst} (see right before \cref{mainlem:horn-filler-trivial}): 
 If $s$ is some set, we write $\overline s$ for the poset generated by $s$. If $i$ is an element of $s$, then $\overline{s}_{-i}$ is the poset $\overline s$ without the set $s$ and without the set $s-i$.
 
 Consider the following diagram in the poset $\mathsf{Sub}(\ordinal{k+1})$:
 \begin{center}
  \begin{tikzpicture}[align=left, node distance=3.0cm]
  \node [](M1) {$\overline{\ordinal{k+1}} - \ordinal{k+1}$};
  \node [below of=M1, node distance = 1.2cm](H) {$\overline{\ordinal{k+1}}_{-0}$}; 
  \node [right of=M1](E) {$\overline{\ordinal{k}}$}; 
  \node [right of=H](M2) {$\overline{\ordinal{k}} - \ordinal{k}$}; 
  
  \draw[->] (M1) to node {} (H);
  \draw[->] (M1) to node {} (E);
  \draw[->] (H) to node {} (M2);
  \draw[->] (E) to node {} (M2);
  \end{tikzpicture}
 \end{center}
 If we apply the functor $\lim_{-} (\Ee{B} \circ U)$ on this square, we get
 \begin{center}
  \begin{tikzpicture}[align=left, node distance=3.0cm]
  \node [](M1) {$M^{\Ee{B}}_{\ordinal {k+1}}$}; 
  \node [below of=M1, node distance = 1.2cm](H) {$\lim_{\overline {\ordinal {k+1}}_{-0}}(\Ee{B} \circ U)$}; 
  \node [right of=M1](E) {$\Ee{B}_{\ordinal{k}}$}; 
  \node [right of=H](M2) {$M^{\Ee{B}}_{\ordinal{k}}$}; 
  
  \draw[->>] (M1) to node {} (H);
  \draw[->>] (M1) to node {} (E);
  \draw[->>] (H) to node {} (M2);
  \draw[->>] (E) to node {} (M2);
  \end{tikzpicture}
 \end{center}
 where the bottom left type is the $0$-th $\ordinal k$-horn as in \cref{mainlem:horn-filler-trivial}.
 By the induction hypothesis, the right vertical fibration is an $(n-k)$-truncated type. By \cref{lem:continuous-functor}, the square is a pullback.
 This means that the left vertical fibration is $(n-k)$-truncated as well, as fibres on the left side are homotopy equivalent to fibres on the right side. 
 
 Consider the composition of fibrations
 \begin{equation}
  \Ee{B}_{\ordinal {k+1}} \fib M^{\Ee{B}}_{\ordinal{k+1}} \fib \lim_{\overline {\ordinal {k+1}}_{-0}}(\Ee{B} \circ U).
 \end{equation}
 Intuitively, the horn is a ``tetrahedron with missing filler and one missing face'', the matching object is the same plus one component which represents this face, and $\Ee{B}_{\ordinal k}$ has, in addition to the face, also a filler of the whole boundary. The filler is really the statement that the ``new'' face equals the canonical one, and we can now make this intuition precise by applying \cref{mainlem:singletons-everywhere}.
 Let us check the conditions:
 \begin{itemize}
  \item  Certainly, we can write the sequence in the form
 \begin{equation}
  \sm{x:X}\sm{x:Y(x)}Z(x,y) \fib \sm{x:X} Y(x) \fib X
 \end{equation}
 (this is given by \cref{lem:sub-gives-fibs}).
  \item The composition is a homotopy equivalence by \cref{mainlem:horn-filler-trivial}.
 \end{itemize}
 Thus, we can assume that $Z(x,y)$ is equivalent to $\id[Y(x)]{y}{y_x}$ for some $y_x$, and thereby of a truncation level that is by one lower than $Y(x)$. But the latter is $(n-k)$ as we have seen above.\footnote{On low levels, we can consider the situation in terms of the presentation in \cref{fig:e-m-ideal}. Here, $y_x$ will be the ``missing face'' that one gets by gluing together the other faces.}
\end{proof}

As a corollary, we get the case for $\ordinal k \jdeq \ordinal{n+2}$:
\begin{corollary} \label{cor:infconstancy-finite}
 Let $B$ be an $n$-type. 
 Then, the fibration
 \begin{equation}
  \Ee{B}_{\ordinal {n+2}} \fib M^{\Ee{B}}_{\ordinal {n+2}}
 \end{equation}
 is a homotopy equivalence. \qed
\end{corollary}

We are now in the position to formulate our result for $n$-types with finite $n$.
Recall from \cref{not:ton} that we write $A \ton {n} B$ for $\Nn_{n}$.
\begin{theorem}[Finite general universal property of the propositional truncation] \label{thm:cor}
 Let $n$ be a fixed number, $-2 \leq n < \infty$. 
 In Martin-L\"of type theory with propositional truncations and function extensionality we can, for any type $A$ and any $n$-type $B$, derive a canonical function
 \begin{equation}
  \big(\proptrunc A \to B\big) \; \to \; \big(A \ton {n+1} B\big)
 \end{equation}
 that is a homotopy equivalence. 
\end{theorem}
\begin{proof}
 Looking at \cref{cor:infconstancy-finite} and at the definition of $\Nn$, as given in \cref{sec:nat-trans}, we see immediately that each $\Nn_{\ordinal {k+1}} \fib \Nn_{\ordinal k}$ with $k \geq n+1$
 is a homotopy equivalence. Thus, using \cref{lem:acyclic_limit}, the Reedy limit $\lim_{\truncCat\deltop}\Nn$ is equivalent to $\Nn_{\ordinal{n+1}}$, and these are $A \toomega B$ and $A \ton {n+1} B$ by definition.
 Similarly, the limit $\lim_{\truncCat\opccat}\NN$ (which we used in the proof of \cref{lem:main}) is homotopy equivalent to the limit over $\truncCat\opccat$ restricted to $\Set{\cob k i | k \leq n+1}$. 
 It is easy to see that the whole proof can be carried out using only finite parts of the infinite diagrams.
 But then, of course, all we need are finitely many nested $\Sigma$-types instead of Reedy $\oppo\omega$-limits, and these automatically exist.
 Further, the only point where we crucially used the judgmental $\eta$-rule for $\Sigma$ is the proof of \cref{thm:result}. In the finite case, however, this is not necessary, as \cref{lem:ignore-trunc-A} is sufficient (similarly, the judgmental $\eta$-rule for $\Pi$-types is not necessary).
 Therefore, the whole proof can be carried out in the standard version of MLTT with propositional truncations.
\end{proof}

\section{Concluding Remarks} \label{sec:conclusions}
 
For any type $B$, we have constructed the equality semi-simplicial type, written $\Ee{B} : \deltop \to \ttfc$, and we have shown that natural transformations from the \emph{trivial} diagram $\Aa{A}$ (the $\ordinal 0$-coskeleton of the diagram constantly $A$) to $\Ee{B}$ correspond to maps $\proptrunc A \to B$.
The construction required us to assume that $\ttfc$ has Reedy $\oppo\omega$-limits.
There are several points that we would like to discuss briefly here, all of which naturally raise further open questions.

First, there are many connections to constructions and results in homotopy theory and the theory of higher topoi, model categories, and quasi-categories.
As we have discussed, for any type $B$ and any inverse category $\I$ that is admissible for $\ttfc$, the constructed equality diagram $\Ee{B} : \I \to \ttfc$ is a Reedy fibrant replacement of the diagram that is constantly $B$. Similarly, the diagram $\Aa{A}$ is a $\ordinal 0$-coskeleton.
One anonymous reviewer has pointed out that \cref{thm:result} is a type-theoretic version of a result on $(\infty,1)$-topoi by Lurie~\cite[Proposition 6.2.3.4]{lurie_topos}.
There are certainly deep connections that have yet to be explored.

Second, we have presented the assumptions of Reedy $\oppo\omega$-limits as a necessary requirement. 
However, we are not aware of a model in which the necessary limits are absent. 
Even though it seems very unlikely, it is in principle possible that these Reedy limits exist in any type-theoretic fibration category automatically.
Assume $A$ and $B$ are some given types.
We do not know whether it is possible to define
the expression $\Nn_{A,B}(n)$ for a \emph{variable} $n$ in HoTT, i.e.\ to give a function $f_{A,B} : \N \to \UU$ (where $\UU$ is a universe) such that the type $f_{A,B}(n)$ is equivalent to $\Nn_{A,B}(n)$ for all \emph{numerals} $n$. 
If this can be done, it should be possible to actually construct what is intuitively an ``infinite $\Sigma$-type'', by asking for all finite approximations with proofs that they fit together, and we could reasonably hope that \cref{thm:result} can be proved in HoTT without any further assumptions.
This has been made precise for the more general case of $M$-types by 
Ahrens, Capriotti and Spadotti~\cite{paolo_nonwellfounded}.
However, we do not expect that such a function $f_{A,B}$ can be defined.
This is at least as hard as defining the equality semi-simplicial type over $B$ as a function $\Ee{B} : \N \to \UU$; this would correspond to the special case where $A$ is the unit type. 
Defining $\Ee{B}$ in this way, however, seems to be as difficult as the famous open problem of defining semi-simplicial types internally as a function $\mathsf{SS} : \N \to \UU_1$ (where $\UU_1$ is a universe that is larger than $\UU$). 
The two problems are identical apart from the fact that the fibres over the matching objects differ.
For $\Ee{B}$, the fibre over a point $m$ of the matching object is $\sm{x:B} \id{\tilde \eta_i(x)}{m}$ as can be seen from~\eqref{eq:define-E-by-fibres} on page~\pageref{eq:define-E-by-fibres}, while for $\mathsf{SS}$, the fibres are constantly the universe $\UU$.
The author does not expect that this makes a real difference in difficulty.
It seems likely that a function $\Ee{B}$ would enable us to talk about coherent equalities so that we could define the function $\mathsf{SS}$, implying that defining $\Ee{B}$ is at least not easier.

Going back a step, while we can prove \cref{thm:cor} internally if $n$ is instantiated with any numeral, we conjecture that it is impossible to prove it for a variable $n$. 
What we think is certainly possible is to write a program in any standard programming language that takes a number $n$ as input and prints out the formalised statement of \cref{thm:cor} (in the syntax of a proof assistant such as Coq or Agda) together with a proof. 
Even in Agda itself, we would be able to define a function which generates the Agda source code of \cref{thm:cor}, for any natural number $n$.
This would provide a solution if we were able to interpret syntax of HoTT in HoTT, which is another famous open problem~\cite{shulman:eating}.

Third, instead of asking whether HoTT allows us to define Reedy fibrant diagrams such as $\Ee{B}$ or $\mathsf{SS}$, we may choose to work in a theory in which we know that it is possible.
Candidates are Voevodsky's HTS (homotopy type system)~\cite{voe_hts}, or the two-level system outlined by Altenkirch, Capriotti and the current author~\cite{altenkirch_twolevels}.
We believe that the results of the current article can be formalised in such settings.

Fourth, is seems obvious to ask whether statements analogous to \cref{thm:result,thm:cor} can be derived for higher truncation operators, written $\trunc n -$~\cite[Chapter 7.3]{HoTTbook}.
A partial result, namely a characterisation of maps $\trunc k A \to B$ if $B$ is $(k+1)$-truncated, have been obtained by Capriotti, Vezzosi and the current author~\cite{capKraVez_elimTruncs}. 

More general results are currently unknown, but we want to conclude with a conjecture.
Assume a type $A$ and an object $\ordinal k$ of $\deltplus$ are given.
We define the \emph{(fibrant) $\ordinal k$-skeleton} of the diagram that is constantly $A$, written $\coskel k A {}$, by giving the fibres over the matching objects:
\begin{equation}
  \coskel k A {\ordinal i} \defeq \begin{cases}
                          \Ee{A}_{\ordinal i} & \mbox{if $\ordinal i \preceq \ordinal k$} \\[4pt]
                          M^{\coskel k A {}}_{\ordinal i} & \mbox{else}.
                         \end{cases}
\end{equation}
Note that, with this definition, the diagram $\Aa{A}$ that we have defined earlier is not \emph{exactly} $\coskel 0 A {}$ for the same reason as for which $\Ee{A}_{\ordinal 0}$ is not \emph{exactly} $A$, but of course, $\coskel 0 A {}$ and $\Aa{A}$ are homotopy equivalent. 
In principle, we could have done the whole proof with $\coskel 0 A {}$ instead of $\Aa{A}$.
Merely for convenience, we have taken advantage of the fact that $\Aa{A}$ is already Reedy fibrant.

For $n \geq -1$, we conjecture that natural transformations from $\coskel {n+1} A {}$ to $\Ee{B}$ correspond to functions $\trunc n A \to B$.
Even more generally, given a higher inductive type $H$, it may be possible to determine a representation of $H$ as a diagram $\repr H : \deltop \to \ttfc$
such that natural transformations from $\repr H$ to $\Ee{B}$ corresponds to functions $H \to B$.
This is very simple for non-recursive higher ``inductive'' types that do not refer to $\refl{}$ or applications of $J$ in their constructors: for example, the circle $\Sn^1$ can be represented with $\repr {\Sn^1}_{\ordinal 0} \jdeq \repr {\Sn^1}_{\ordinal 1} \jdeq \unit$ and $\repr {\Sn^1}_{\ordinal {n+2}} \jdeq \emptyt$, while the suspension of $A$ can be realised as $\repr{\Sigma A}_{\ordinal 0} \jdeq \bool$, $\repr{\Sigma A}_{\ordinal 1} \jdeq A$, and $\repr{\Sigma A}_{\ordinal {n+2}} \jdeq \emptyt$.
If this turns out to work for a larger class of higher inductive types, it may be understood as a type-theoretic version of the \emph{homotopy hypothesis} which has so far suffered from the difficulty of formulating the coherences of categorical laws~\cite{hirschowitz_wild}.

\vspace*{0.5cm}

\subparagraph*{\textbf{Acknowledgements.}}

First of all, I want to thank Paolo Capriotti. It was him who pointed out to me that what I was doing could be formulated nicely in terms of diagrams over inverse categories, and he helped me to transfer my original definition of the equality semi-simplicial type to this setting.
Numerous discussions with him helped me greatly to understand Shulman's work and related concepts.
He, as well as Christian Sattler and the anonymous reviewers, has enabled me to understand connections with homotopy theory that I had not been aware of.
Especially the reviewers' reports have been helpful for the improvements that have been incorporated in this article.
Steve Awodey, as examiner of my Ph.D.\ thesis, has given me very useful feedback on this work as well.
I am grateful to Nils Anders Danielsson, who has formalised the propositions of \cref{sec:motivation}, and who has pointed out an issue with \cref{lem:ignore-trunc-A} in an earlier version of this article.

I further want to thank Thorsten Altenkirch for his constant general support, Mart{\'i}n Escard{\'o} for many discussions on related questions, and Michael Shulman for making his LaTeX macros publicly available.

\bibliographystyle{plain}
\bibliography{refs_infconstancy}

\end{document}